\pdfoutput=1
\RequirePackage{ifpdf}
\ifpdf 
\documentclass[pdftex]{sigma}
\else
\documentclass{sigma}
\fi

\numberwithin{equation}{section}

\newcommand{\tp}{$(2{+})$}
\newcommand{\nat}{\mathbb{N}}
\newcommand{\intn}{\mathbb{Z}}
\newcommand{\realn}{\mathbb{R}}
\newcommand{\ratn}{\mathbb{Q}}
\newcommand{\divides}[2]{#1\vert #2}
\newcommand{\Gp}{\Gamma_0(2)^+}
\newcommand{\gindex}[1]{\hbox{\rm Index}\left( #1 \right)}

\newtheorem{Theorem}{Theorem}[section]
\newtheorem{Corollary}[Theorem]{Corollary}
\newtheorem{Lemma}[Theorem]{Lemma}
\newtheorem{Proposition}[Theorem]{Proposition}
 { \theoremstyle{definition}
\newtheorem{Definition}[Theorem]{Definition}

\newtheorem{Remark}[Theorem]{Remark} }

\begin{document}
\allowdisplaybreaks

\newcommand{\arXivNumber}{1710.01071}

\renewcommand{\thefootnote}{}

\renewcommand{\PaperNumber}{060}

\FirstPageHeading

\ShortArticleName{$(2+)$-Replication and the Baby Monster}

\ArticleName{$\boldsymbol{(2+)}$-Replication and the Baby Monster\footnote{This paper is a~contribution to the Special Issue on Moonshine and String Theory. The full collection is available at \href{https://www.emis.de/journals/SIGMA/moonshine.html}{https://www.emis.de/journals/SIGMA/moonshine.html}}}

\Author{Chris CUMMINS~$^\dag$ and Rodrigo MATIAS~$^\ddag$}

\AuthorNameForHeading{C.~Cummins and R.~Matias}

\Address{$^\dag$~Department of Mathematics and Statistics, Concordia University,\\
\hphantom{$^\dag$}~1455 de Maisonneuve Blvd Ouest, Montr\'eal, H3G 1M8, Qu\'ebec, Canada}
\EmailD{\href{mailto:chris.cummins@concordia.ca}{chris.cummins@concordia.ca}}

\Address{$^\ddag$~Departamento de Matem\'atica, Faculdade de Ci\^{e}ncias e Tecnologia,\\
\hphantom{$^\ddag$}~Universidade de Coimbra, Portugal}
\EmailD{\href{mailto:rmatias@mat.uc.pt}{rmatias@mat.uc.pt}}

\ArticleDates{Received October 04, 2017, in final form May 31, 2018; Published online June 16, 2018}

\Abstract{The definitions of replicable and completely replicable functions are intimately related to the Hecke operators for the modular group. We define the notions of ``$(2+)$-replicable'' and ``completely $(2+)$-replicable'' functions by considering the Hecke operators for $\Gamma_0(2)^+$. We prove that the McKay--Thompson series for $2\cdot\mathbb{B}$, as computed by H\"ohn, are completely $(2+)$-replicable.}


\Classification{11F22; 11F25}

\renewcommand{\thefootnote}{\arabic{footnote}}
\setcounter{footnote}{0}

\section{Introduction}

The monstrous moonshine conjectures of Conway and Norton \cite{CN}, led, via rapid developments in VOA theory \cite{B2, FLM} and generalized Kac--Moody algebras \cite{B3}, to Borcherds' proof of the conjectures \cite{B}. In this paper our aim is to generalize one aspect of this proof to the case of the baby monster, $\mathbb{B}$, namely the connection between an appropriate generalization of ``complete replication'' of modular functions and the power map structure of $2\cdot \mathbb{B}$.

To explain this connection, recall the monstrous case.
For this, Norton \cite{N2}, introduced the idea of replicable and completely replicable functions. A formal $q$-series $f=q^{-1}+c_1q+ c_2q^2 + \cdots$ with rational coefficients is said to be replicable if and only if there exist formal $q$-series $f^{(a)} = q^{-1} + c_1^{(a)}q + c_2^{(a)}q^2 + \cdots $, $a=1,2,3,\dots$ such that
\begin{gather}
\sum_{\substack{ad=n \\ 0\leq b<d}} f^{(a)}\left(\frac{a\tau+b}{d}\right)=P_{n,f}(f(q)), \label{repeqns}
\end{gather}
where $P_{n,f}(X)$ is the $n$-th Faber polynomial of~$f$. This is the unique polynomial such that $P_{n,f}(f(q))-\frac{1}{q^n}$ has only positive powers of~$q$.

The series $f^{(a)}$ are called the replicates of $f$. Norton conjectured that $f$ is replicable if and only if it either has the form $q^{-1}+ cq$ or is the $q$ expansion of a Haupmodul of certain congruence subgroups. A completely replicable $f$ is a replicable function for which $f^{(a)}$ is also replicable for all positive integers $a$. Koike showed that a certain set of modular functions, which includes those occurring in monstrous moonshine, are completely replicable and that their replicates agree with those predicted by the monstrous moonshine conjectures.

The Hauptmoduls for genus zero groups with rational integer coefficients are known to be replicable functions~\cite{CuN}. The genus zero congruence subgroups of ${\rm PSL}(2,\realn)$ have been classified in \cite{ChCm} and the completely replicable functions (with rational integer coefficients) were computed in~\cite{ACMS}.

The connection with the monster, $\mathbb{M}$, is that the trace functions of the action of $\mathbb{M}$ on $V^\natural$ are completely replicable as a consequence of the twisted denominator formula for the monster Lie algebra \cite{B, CN,FLM2}.

Thus we have the surprising fact that the relatively simple characterization of completely replicable
functions in a sense captures the power map structure of the monster.

These facts have been generalized. In particular Norton \cite{N} introduced generalized moonshine
by associating functions to pairs of commuting elements in $\mathbb{M}$. There has been considerable
progress in understanding this phenomenon and Carnahan has announced a proof of the
generalized moonshine conjectures \cite{carn,carn2,carn4}.

\looseness=-1 Here we take a complementary approach which aims to understand and extend complete replicability to the case of the baby monster. The idea is that the replication equations (\ref{repeqns}) were discovered by Conway and Norton by modifying the Hecke operators for ${\rm PSL}(2,\intn)$ to take into account the conjectured trace functions on classes of $\mathbb{M}$ other than the class of the identity. Instead we start here with the Hecke operators for $\Gp$ and then look for a natural way to introduce replication and complete replication in this context (see \cite{CN} for notation). The motivation is that $2\cdot\mathbb{B}$ is the centralizer in $\mathbb{M}$ of an element from the conjugacy class labeled $2A$ in the atlas and in monstrous moonshine the class $2A$ is associated with the Hauptmodul for $\Gamma_0(2)^+$.

This is done in Section \ref{heckeandrep} and we find a different form of replicability and prove that it reflects the power map structure in $2\cdot\mathbb{B}$. We call it \tp-replicability as it is motivated by the Hecke operators of $\Gp$ (which is denoted by \tp in Conway and Norton's notation). Although the functional equations we use are implicit in the work of Carnahan \cite{carn,carn2, carn4} and Borcherds \cite[Section~10]{B}, we find that they have a natural interpretation as \tp-replication identities.

Our main motivation for this approach is that H\"ohn \cite{H} has given a proof of the baby moonshine conjectures which closely parallels Borcherds' monstrous proof, but with a method which uses replicability rather than complete replicability. Introducing complete \tp-replicability allows us to modify part of H\"ohn's proof so as to clone Borcherds' use of complete replicability in his monstrous proof. To do this, in Section \ref{CRHauptmoduls} we show that a large class of Hauptmoduls are completely \tp-replicable by working out candidate replicates. This class includes all but~13 of the Hauptmoduls which occur in moonshine for the baby monster. In Section~\ref{mahler} we prove that a completely \tp-replicable function is completely determined by the first $5$ coefficients of the function and its replicates. In the final Section \ref{BMandrep} we prove that the baby monster McKay--Thompson series are completely \tp-replicable with \tp-replicability respecting the power map structure in $2\cdot\mathbb{B}$. This is sufficient to establish that McKay--Thompson series are Hauptmoduls once their first 5 coefficients are calculated, except for the 13 exceptions noted above. For these 13 cases our proof is similar that of H\"ohn, but simplified somewhat since we have additional recurrence relations satisfied by the McKay--Thompson series. This establishes, incidentally, that these exceptions are also \tp-replicable~-- however we have not found a way to include them in the cases covered in Section~\ref{CRHauptmoduls}.

More generally, it is our belief that the notions of \tp-replicability and complete \tp-replicability will lead to a better understanding of the connections between moonshine and Hecke operators.

\section{Hecke operators and replication}\label{heckeandrep}
Some references for the background material needed for this section are: Chapter~5 in~\cite{DSh} for Hecke operators and~\cite{CN} for the basics of moonshine and the concept of replicability.

\looseness=-1 Let $G$ be a discrete subgroup of ${\rm PSL}(2,\realn)$ which is commensurable with ${\rm PSL}(2,\intn)$. We consider $\mathcal{F}(G)$ the set of functions meromorphic on the upper half plane and at cusps that are invariant under the action of $G$. For each element $\alpha\in {\rm PSL}(2,\realn)$ in the commensurator of $G$ we define a~Hecke operator $ T_\alpha\colon \mathcal{F}(G)\longrightarrow \mathcal{F}(G)$ in the following way. Consider a decomposition $G\alpha G= \bigcup\limits_{j=1}^n G\gamma_j$ as a disjoint union, guaranteed to be finite as $\alpha$ is in the commensurator of $G$. Define
\begin{gather*}
T_\alpha f(z)= \sum_{j=1}^n f(\gamma_j z).
\end{gather*}

An important example is given by the Hecke operators of $G=\Gamma = {\rm PSL}(2,\intn)$ (see, for example, \cite[pp.~60--63]{Shi}. Let $\Delta = \{ \alpha \in M_2(\intn) \,|\, \det(\alpha) > 0\}$ and let ${ \tilde{\mathbb T}}_m$ be the sum over all double cosets $\Gamma \alpha \Gamma$ with $\alpha\in \Delta$ and $\det(\alpha) = m$. Also define ${\mathbb T}_{a,d}$ to be the double coset $\Gamma \left[\begin{smallmatrix} a&0\\0&d\end{smallmatrix}\right]\Gamma $ and set ${\mathbb T}_{m} = {\mathbb T}_{1,m}$. From the theory of elementary divisors we then have formula ${\mathbb {\tilde T}}_m = \sum {\mathbb T}_{a,d}$ where the sum is over all positive $a$ and $d$ such that $ad = m$ such that $a$ divides $d$. As explained above, the Hecke action of these (distinct) double cosets on modular functions is obtained by expressing the double cosets as a union of left cosets (by a slight abuse of notation we use the same symbol for both a sum of double cosets and the corresponding operator). For example, if~$f$ is a modular function for $\Gamma$ then
\begin{gather}
 \tilde{\mathbb T}_4f(\tau) = {\mathbb T}_{1,4}f(\tau) + {\mathbb T}_{2,2}f(\tau)= f(4\tau) + f\left(\frac{2\tau}{2}\right) + f\left(\frac{2\tau+1}{2}\right) \nonumber\\
\hphantom{\tilde{\mathbb T}_4f(\tau) =}{} + f\left(\frac{\tau}{4}\right) + f\left(\frac{\tau+1}{4}\right) + f\left(\frac{\tau+2}{4}\right) + f\left(\frac{\tau+3}{4}\right). \label{t4example}
\end{gather}

The relationship to the Conway--Norton replication formulas of ``classical'' moonshine is as follows. When $f=f_g$ is a Thompson--McKay series for an element $g$ in the monster, the Conway--Norton replication formula corresponding to (\ref{t4example}) is
\begin{gather}
 f^{(4)}(4\tau) + f^{(2)}\left(\frac{2\tau}{2}\right) + f^{(2)}\left(\frac{2\tau+1}{2}\right) + f\left(\frac{\tau}{4}\right)\nonumber\\
 \qquad{} + f\left(\frac{\tau+1}{4}\right) + f\left(\frac{\tau+2}{4}\right) + f\left(\frac{\tau+3}{4}\right)=P_{4,f}(f) , \label{t4rep}
\end{gather}
where $P_{4,f}(f)$ is the 4th.~Faber polynomial of $f$ and the functions $f^{(4)}=f_{g^4}$ and $f^{(2)}=f_{g^2}$ are the Thompson--McKay series for the monstrous elements $g^4$ and $g^2$ respectively. Thus the left side of this replication formula is a modification of the Hecke action of $ {\tilde{\mathbb T}}_4$. So in this sense the Conway--Norton replication identities say that the result of a modified Hecke operator acting on a Thompson--McKay series is equal to a Faber polynomial in that series.

We will show that a similar phenomenon occurs if we replace the monster by $2\cdot\mathbb{B}$ and ${\rm PSL}(2,\intn)$ by $\Gp$. For the $\Gp$ case we set
\begin{gather*}
\Delta_2 = \left\{ \begin{pmatrix}a & b\\ 2c & d\end{pmatrix} |\, a,b,c,d\in\intn,\, ad-2bc>0 \right\}
\end{gather*}
and let
\begin{gather}\label{Tm}
 {\bf \tilde{T}}_m = \sum_{\substack{\alpha \in \Delta_2\\ \det(\alpha) = m}} \Gp \alpha \Gp.
\end{gather}
Also define ${\bf T}_{a,d}$ to be the double coset $\Gp \left[\begin{smallmatrix} a&0\\0&d\end{smallmatrix}\right]\Gp $ and ${\bf T}_{m} = {\bf T}_{m,1}$. As we will see ${\bf \tilde{T}}_m = \sum {\bf T}_{a,d}$ where in this case the sum is over all positive integers~$a$ and~$d$ such that~$a$ divides~$d$ and either $ad=m$ or $ad=m/2$ (the second option being absent if $m$ is odd).

As for the modular group we have an action on functions. For example
\begin{gather}
 {\bf \tilde{T}}_2f(\tau) = {\bf T}_{1,2}f(\tau) + {\bf T}_{1,1}f(\tau)\nonumber \\
\hphantom{{\bf \tilde{T}}_2f(\tau)}{} = f(2\tau) + f(\tau) + f\left(\tau +\frac{1}{2}\right) + f\left(\frac{\tau}{2}\right) + f\left(\frac{\tau+1}{2}\right), \label{BabyT2}
\end{gather}
where we have used Proposition \ref{cosetdecomp} below. As we will see, there is a ``replication formula'' for $2\cdot\mathbb{B}$ which involves a modification of (\ref{BabyT2}) in the same way that (\ref{t4rep}) involves a modification of~(\ref{t4example}). It is
\begin{gather}
 f^{[2]}(2\tau) + f^{[\sqrt{2}]}(\tau) + f^{[\sqrt{2}]}\left(\tau +\frac{1}{2}\right) + f\left(\frac{\tau}{2}\right) + f\left(\frac{\tau+1}{2}\right)=P_{2,f}(f). \label{Babyrep}
\end{gather}

The novelty here is not in the existence of these identities, as mentioned above they are implicit in previous work. Nor is it the existence of the \tp-replicates, which, as we shall see, are trace functions on appropriate VOA-modules. Rather we believe that the key points are firstly that a form of complete replication is restored by introducing ``half step'' replicates as in (\ref{Babyrep}). So, for example, $\big(f^{[\sqrt{2}]}\big)^{[\sqrt{2}]} = f^{[2]}$ (which explains our chosen normalization for the exponents of \tp-replicates). The absence of complete replicability is a complicating factor for groups other than the monster and introducing complete \tp-replicability means that we can give a modified version of H\"ohn's $2\cdot\mathbb{B}$ proof which is closer to Borcherds' monstrous proof. The second key point is that this approach emphasizes a close connection between replication identities and Hecke algebras of groups other than ${\rm PSL}(2,\intn)$. A connection which, we believe, may prove fruitful.

The purpose of the next section is to define \tp-replicability based on the Hecke operators for $\Gp$. We then define complete \tp-replicability and we will see that the set of completely \tp-replicable functions includes the McKay--Thompson series for the baby monster group.

\subsection[Hecke operators for $\Gp$]{Hecke operators for $\boldsymbol{\Gp}$}\label{gamma2}

In this subsection we find expressions for the Hecke operators of $\Gp$. These will be used in Section~\ref{2Arepl} to motivate the definition of \tp-replication.

Let $m$ be a positive integer. We define the following sets
\begin{gather}
M_1^m = \left\{ \begin{bmatrix} x & y \\ 0 & z\end{bmatrix} |\,
 xz=m,\, 0\leq y < z,\, \gcd(x,y,z) = 1,\, x\ \hbox{odd} \right\}, \nonumber\\
S_1^m = \left\{ \begin{bmatrix} x & y \\ 0 & z\end{bmatrix} |\,
 xz=m,\, 0\leq y < z,\, \gcd(x,y,z) = 1,\, z\ \hbox{odd} \right\}, \nonumber \\
S_2^m = \left\{ 2^{-1/2}\begin{bmatrix} x & y \\ 0 & z\end{bmatrix} |\,
 xz=2m,\, 0\leq y < z,\, \gcd(x,y,z) = 1,\, x,z\ \hbox{even} \right\},\nonumber \\
 M_2^m = S_1^m \cup S_2^m, \qquad
 M^m = M_1^m \cup M_2^m.\label{MSdefinitions}
\end{gather}

The main result of this subsection is the following

\begin{Proposition}\label{cosetdecomp}
 Let $m$ be an even positive integer then we have the following
\begin{enumerate}\itemsep=0pt
 \item[$1)$] $ \Gp\left[\begin{smallmatrix} 1 & 0 \\ 0 & m \end{smallmatrix}\right] \Gamma_0(2)
 = \bigcup_{\gamma\in M^m_1}\Gp \gamma $,
 \item[$2)$] $ \Gp\left[\begin{smallmatrix} m & 0 \\ 0 & 1 \end{smallmatrix}\right] \Gamma_0(2)
 = \bigcup_{\gamma\in M^m_2}\Gp \gamma $,
 \item[$3)$] $ \Gp\left[\begin{smallmatrix} m & 0 \\ 0 & 1 \end{smallmatrix}\right] \Gp
 =
 \Gp\left[\begin{smallmatrix} 1 & 0 \\ 0 & m \end{smallmatrix}\right] \Gp =
 \bigcup_{\gamma\in M^m}\Gp \gamma$.
\end{enumerate}
\end{Proposition}

The proof will proceed by several lemmas. Until Proposition \ref{oddcosetdecomp}, $m$ will always be a positive even integer.

\begin{Lemma} \label{distinctlemma} The cosets $\Gp\gamma$ for $\gamma$ in $M^m$ are distinct.
\end{Lemma}
\begin{proof} For any $\gamma$ and $\gamma'$ in $M^m$ we have that $r = \gamma'\gamma^{-1}$ fixes $\infty$. So if $r$ is in $\Gp$ it must be of the form $\pm\left[\begin{smallmatrix} 1 & k \\ 0 & 1\end{smallmatrix}\right]$. A~short calculation now shows that $\gamma=\gamma'$.
 \end{proof}

In what follows, for a positive integer $n$ we define $\psi(n)=n \prod\limits_{\substack{p|n \\ p\ \text{prime}}} \big(1-\frac{1}{p}\big)$.

 \begin{Lemma} \label{indexlemma} Let $m = 2^ab$ with $a> 0$ and $b$ odd. Then we have
 \begin{gather*}
 \big\vert M_1^m \big \vert = 2^a\psi(b)
 = \gindex{\Gamma_0(2),\Gamma_0(2)\cap
 \begin{bmatrix} 1 & 0 \\0 & m\end{bmatrix}^{-1}
 \Gp
 \begin{bmatrix} 1 & 0 \\0 & m\end{bmatrix}}.
 \end{gather*}
 \end{Lemma}

\begin{proof} First observe that each element of $M^m_1$ has the form $\left[\begin{smallmatrix} d & y \\ 0 & 2^a(b/d)\end{smallmatrix}\right]$ where~$d$ is a divisor of~$b$ (and hence odd) and $0\leq y < 2^a(b/d)$. The number of such matrices is $2^a$ times the number of matrices of the form $\left[\begin{smallmatrix} d & y \\ 0 & (b/d)\end{smallmatrix}\right]$, but where $0\leq y < (b/d)$. By the standard theory of Hecke operators for the modular groups, the number of these matrices is $\psi(b)$, the index of $\Gamma_0(b)$ in the modular group, and so the first equation follows.

For the second equation, if $w$ is in $\Gp$ we have that $\left[\begin{smallmatrix} 1 & 0 \\0 & m\end{smallmatrix}\right]^{-1} w \left[\begin{smallmatrix} 1 & 0 \\0 & m\end{smallmatrix}\right]$ is in $\Gamma_0(2)$ if and only if $w$ is in $\Gamma_0(2m)$. Since conjugation preserves areas of fundamental domains, it preserves indexes and so the required index is that of $\Gamma_0(2m)$ in $\Gamma_0(2)$ which is $\psi(2m)/3 = \psi(2^{a+1}b)/3 = 2^a\psi(b)$ as required.
 \end{proof}

\begin{Lemma} \label{closurelemma} The cosets $\Gp\gamma$ for $\gamma$ in $M_m^1$ are closed under right multiplication by elements of $\Gamma_0(2)$.
 \end{Lemma}

\begin{proof} The group $\Gamma_0(2)$ is generated by $\left[\begin{smallmatrix} 1 & 1 \\ 0 & 1 \end{smallmatrix}\right]$ and $\left[\begin{smallmatrix} 1 & 0 \\ 2 & 1 \end{smallmatrix}\right]$. Closure under right multiplication by the former is easy to verify.

 For the latter we have
 \begin{gather*}
 \begin{bmatrix} x & y \\ 0 & z \end{bmatrix}
 \begin{bmatrix} 1 & 0 \\ 2 & 1 \end{bmatrix} =
 \begin{bmatrix} x+2y & y \\ 2z & z \end{bmatrix}.
 \end{gather*}

If $g=\gcd(x+2y,2z)$ then $g$ is odd since $x$ is odd. Let $a$ and $b$ be such that $a2z+b(x+2y)=g$. Then $r =
\left[\begin{smallmatrix} b & a \\ -(2z/g) & (x+2y)/g \end{smallmatrix}\right]$ is in $\Gamma_0(2)$ and we have
 \begin{gather*}
 r\begin{bmatrix} x+2y & y \\ 2z & z \end{bmatrix} = \begin{bmatrix} g & (g-bx)/2 \\ 0 & m/g \end{bmatrix}.
 \end{gather*}

Up to left multiplication by a translation this last element is in $M_1^m$, as required.
 \end{proof}

 \begin{Lemma} \label{part1lemma}
 Part~$1$ of Proposition~{\rm \ref{cosetdecomp}} holds:
 \begin{gather*}
 \Gp\begin{bmatrix} 1 & 0 \\ 0 & m \end{bmatrix} \Gamma_0(2)
 = \bigcup_{\gamma\in M^m_1}\Gp \gamma.
 \end{gather*}
 \end{Lemma}
\begin{proof} By Lemma \ref{closurelemma} and the fact that $\left[\begin{smallmatrix} 1 & 0 \\ 0 & m \end{smallmatrix}\right]$ is in $M_1^m$ we can conclude that $\bigcup\limits_{\gamma\in M^m_1}\Gp \gamma$ is a union of $\Gp - \Gamma_0(2)$ double cosets which includes the double coset $ \Gp\left[\begin{smallmatrix} 1 & 0 \\ 0 & m \end{smallmatrix}\right] \Gamma_0(2)$. According to \cite[Proposition~3.1]{Shi}, if $G_1$ and $G_2$ are commensurable subgroups of a group $G$ then for each $g$ in $G$ the double coset $G_1gG_2$ is a union of disjoint single cosets $\bigcup\limits_{i=1}^d G_1 g_i$ where $d$ is the index of $G_2\cap g^{-1}G_1g$ in $G_2$. Using this result and Lemma~\ref{indexlemma}, the double coset $ \Gp\left[\begin{smallmatrix} 1 & 0 \\ 0 & m \end{smallmatrix}\right] \Gamma_0(2)$ is a union of $\vert M_1^m\vert$ distinct left $\Gp$ cosets. By Lemma~\ref{distinctlemma} the left cosets defined by elements of $M_1^m$ are distinct and so we have equality.
 \end{proof}

\begin{Lemma} \label{equalsizelemma} $\vert M_1^m\vert = \vert M_2^m\vert$.
 \end{Lemma}
\begin{proof} With $m=2^ab$ as above, we have to show $\vert M_2^m\vert = 2^a\psi(b)$. Recall that $M_2^m = S_1^m\cup S_2^m$ with $S_1^m$ and $S_2^m$ defined as in (\ref{MSdefinitions}). We will find the sizes of these two sets.

A typical element of $S_1^m$ has the form $\left[\begin{smallmatrix} 2^a(b/s) & y \\ 0 & s \end{smallmatrix}\right] $ where $s$ is a divisor of $b$ and hence odd. The elements of this matrix are coprime if and only if the elements of the matrix $\left[\begin{smallmatrix} (b/s) & y \\ 0 & s \end{smallmatrix}\right] $ are coprime and there is a one to one correspondence between these two sets of matrices. But the number of matrices of the latter form is $\psi(b)$ and so this is the size of $S_1^m$.

To count the elements of $S_2^m$ consider a typical element, which has the form $2^{-1/2}\left[\begin{smallmatrix} (2m/z) & y \\ 0 & z \end{smallmatrix}\right] $ with $z$ an even divisor of $2m$ such that $2m/z$ is even and~$y$ is such that $0\leq y < z$ and the entries are coprime.

If we ignore the condition that $z$ and $2m/z$ are even, there are $\psi(2m)$ such matrices. But this over-counts by matrices of the form $2^{-1/2}\left[\begin{smallmatrix} b/s & y \\ 0 & 2^{a+1}s \end{smallmatrix}\right] $ of
which there are $2^{a+1}\psi(b)$ (as in the proof of Lemma \ref{indexlemma}), and also matrices of the form $2^{-1/2}\left[\begin{smallmatrix} 2^{a+1}b/s & y \\ 0 & s \end{smallmatrix}\right] $ of which there are $\psi(b)$ (as in case of $S_1^M$ above). This gives $\vert S_2^m\vert = \psi(2^{a+1}b) - 2^{a+1}\psi(b)-\psi(b) = 3\times2^{a}\psi(b) - 2^{a+1}\psi(b) - \psi(b) = 2^a\psi(b)-\psi(b)$. So finally $\vert M_2^m\vert = 2^a\psi(b)$ as required.
 \end{proof}

 Let $w_2 = 2^{-1/2}\left[\begin{smallmatrix} 0 & -1 \\ 2 & 0\end{smallmatrix}\right]$.
 \begin{Lemma}\label{w2lemma}
 \begin{gather*}
 \bigcup_{\gamma\in M_1^m} \Gp\gamma = \bigcup_{\gamma\in M_2^m} \Gp\gamma w_2.
 \end{gather*}
 \end{Lemma}
 \begin{proof} By Lemmas \ref{distinctlemma} and \ref{equalsizelemma} it suffices to show that each $\gamma w_2$ with $\gamma$ in $M_2^m$ is equal, up to left multiplication by elements of $\Gp$, to an element of $M_1^m$.

 {\bf Case I:} Suppose $\gamma
 = \left[\begin{smallmatrix} x & y \\ 0 & z\end{smallmatrix}\right]$ is in $S_1^m$.
 Then $w_2^{-1}\gamma w_2 =
 \left[\begin{smallmatrix} z & 0 \\ -2y & x\end{smallmatrix}\right]$. Let
 $g=\gcd(z,2y)$ which is odd since $z$ is odd, and let $a$ and $b$ be such that
 $az-b2y =g$. Then $r=
\left[\begin{smallmatrix} a & b \\ 2y/g & z/g \end{smallmatrix}\right]$ is in
 $\Gamma_0(2)$ and $r
 \left[\begin{smallmatrix} z & 0 \\ -2y & x\end{smallmatrix}\right] =
 \left[\begin{smallmatrix} g & bx \\ 0 & m/g\end{smallmatrix}\right]$ which, up
 to a left translation, is in $M_1^m$.

 {\bf Case II:} Suppose $\gamma
 = 2^{-1/2}\left[\begin{smallmatrix} x & y \\ 0 & z\end{smallmatrix}\right]$ is in $S_2^m$.
 Then $\gamma w_2 =
 \left[\begin{smallmatrix} y & -x/2 \\ z & 0\end{smallmatrix}\right]$. Let
 $g=\gcd(z,y)$ which is odd since $y$ is odd, and let $a$ and $b$ be such that $az+by =g$ with $a$ even (which is possible since $y/g$
 is odd). Then $r=
 \left[\begin{smallmatrix} b & a \\ -z/g & y/g\end{smallmatrix}\right]$ is in
 $\Gamma_0(2)$ and $r
 \left[\begin{smallmatrix} y & -x/2 \\ z & 0\end{smallmatrix}\right] =
 \left[\begin{smallmatrix} g & -bx/2 \\ 0 & m/g\end{smallmatrix}\right]$ which again, up to a left translation, is in $M_1^m$.
 \end{proof}

 \begin{proof}[Proof of Proposition \ref{cosetdecomp}] Part 1 of the Proposition was proved in Lemma~\ref{part1lemma}.

Part 2 follows from
 \begin{gather*}
 \bigcup_{\gamma\in M_2^m} \Gp\gamma =
 \bigcup_{\gamma\in M_1^m} \Gp\gamma w_2 =
 \Gp\begin{bmatrix} 1 & 0 \\ 0 & m \end{bmatrix} \Gamma_0(2)w_2 \\
 \hphantom{\bigcup_{\gamma\in M_2^m} \Gp\gamma}{} =
 \Gp w_2^{-1}\begin{bmatrix} 1 & 0 \\ 0 & m \end{bmatrix}w_2 \Gamma_0(2) =
 \Gp\begin{bmatrix} m & 0 \\ 0 & 1 \end{bmatrix} \Gamma_0(2),
 \end{gather*}
where the first equality follows from Lemma \ref{w2lemma} and the fact that $w_2^2 = -1_2$.

 Finally, part 3 follows from
 \begin{gather*}
 \Gp\begin{bmatrix} 1 & 0 \\ 0 & m \end{bmatrix} \Gp =\left(\Gp\begin{bmatrix} 1 & 0 \\ 0 & m \end{bmatrix} \Gamma_0(2)\right)\bigcup \left(\Gp\begin{bmatrix} 1 & 0 \\ 0 & m \end{bmatrix} \Gamma_0(2)w_2\right) \\
\hphantom{\Gp\begin{bmatrix} 1 & 0 \\ 0 & m \end{bmatrix} \Gp }{}
=\left(\Gp\begin{bmatrix} 1 & 0 \\ 0 & m \end{bmatrix} \Gamma_0(2)\right)\bigcup\left(\Gp\begin{bmatrix} m & 0 \\ 0 & 1 \end{bmatrix} \Gamma_0(2)\right) \\
\hphantom{\Gp\begin{bmatrix} 1 & 0 \\ 0 & m \end{bmatrix} \Gp }{} =\bigg(\bigcup_{\gamma\in M_1^m} \Gp\gamma\bigg) \bigcup \bigg( \bigcup_{\gamma\in M_2^m} \Gp\gamma \bigg)
=\bigcup_{\gamma\in M^m} \Gp\gamma,
 \end{gather*}
where we have used parts~1 and~2.
 \end{proof}

For the case where $m$ is odd we have from the definition (\ref{MSdefinitions}) above that
\begin{gather*}\label{Modddefinition}
M^m = \left\{ \begin{bmatrix} x & y \\ 0 & z\end{bmatrix}\,|\,
xz=m,\, 0\leq y < z,\, \gcd(x,y,z) = 1\right\}.
\end{gather*}

We then have the following:

\begin{Proposition}\label{oddcosetdecomp}
 Let $m$ be an odd positive integer then
\begin{gather*}
\Gp\begin{bmatrix} 1 & 0 \\ 0 & m \end{bmatrix} \Gamma_0(2) =
\Gp\begin{bmatrix} m & 0 \\ 0 & 1 \end{bmatrix} \Gamma_0(2) = \Gp\begin{bmatrix} m & 0 \\ 0 & 1 \end{bmatrix} \Gp \\
\hphantom{\Gp\begin{bmatrix} 1 & 0 \\ 0 & m \end{bmatrix} \Gamma_0(2)}{} =
\Gp\begin{bmatrix} 1 & 0 \\ 0 & m \end{bmatrix} \Gp = \bigcup_{\gamma\in M^m}\Gp \gamma.
\end{gather*}
\end{Proposition}

\begin{proof} The proof is similar, but simpler, than the proof of Proposition \ref{cosetdecomp} and we omit it.
\end{proof}

We now wish to introduce Hecke operators for $\Gp$. In the introduction to this section, we defined Hecke operators by starting with double cosets and then writing these double cosets as single cosets. We will find it convenient here to reverse this ordering and start by defining operators ${\bf \tilde{T}}_m$ and ${\bf T}_m$ by their actions and then deducing their characterizations as double cosets given in~(\ref{Tm}). Moreover, to further simplify the presentation we only consider actions on functions (i.e., weight zero modular forms) so that the action of ${\bf T}_{a,d}$ is equal to that of ${\bf T}_{a/d,1}$ which, as mentioned earlier, we write as ${\bf T}_{a/d}$. Thus we make the following definitions:

\begin{Definition} \label{HeckeOps}Let $m$ be a positive integer
\begin{enumerate}\itemsep=0pt
 \item[1)]
 ${\bf {T}}_m f(\tau)=
\displaystyle \sum\limits_{\gamma \in M^m} f(\gamma(z))
 $,
 \item[2)]
 $
 {\bf \tilde{T}}_m f(\tau)=
 \begin{cases}
 \displaystyle \sum\limits_{\substack{xz =m\\0\leq y < z}} f\left(\frac{x\tau + y}{z}\right),& m \ \hbox{\rm odd},\\
 \displaystyle \sum_{\substack{xz =m\\0\leq y < z}} f\left(\frac{x\tau + y}{z}\right)+\sum_{\substack{xz =2m\\x,z\ {\rm even}\\0\leq y < z}} f\left(\frac{x\tau + y}{z}\right),& m \ \hbox{\rm even}.
 \end{cases}
$
\end{enumerate}

\end{Definition}By Proposition \ref{cosetdecomp} this definition of the action ${\bf T}_m$ is the same as the action of the double coset $\Gp \left[\begin{smallmatrix} m&0\\0&1\end{smallmatrix}\right]\Gp$ as required.
Moreover, if we observe that every element of~$\Delta_2$ is, up to left multiplication by an element of $\Gp$, equivalent to one of the matrices which occur on the right side of part~2 of Definition~\ref{HeckeOps} and that these matrices represent distinct left cosets of~$\Gp$, then it is straightforward to verify that this definition of~$\tilde{T}_m$ is equivalent to that give in~(\ref{Tm}).

The relationship between ${\bf T}_m$ and ${\bf \tilde{T}}_m$ is given by the following:

\begin{Proposition}\label{HeckeFormula}
 Let $m=2^\alpha \beta$ be a positive integer where $\beta$ is odd, then
\begin{gather*}
 {\bf \tilde{T}}_m = \sum_{\substack{\divides{r^2}{ m}\\r\ {\rm odd}}}
 \sum_{i=0}^\alpha {\bf T}_{m/(r^2 2^i)}.
\end{gather*}
\end{Proposition}
\begin{proof}
 Define first the following operators
\begin{alignat*}{3}
 & {\bf R}_m f(\tau) = \sum_{\substack{xz = m\\0\leq y <z}} f\left(\frac{x\tau +y}{z}\right),\qquad &&
 {\bf D}_m f(\tau) = \sum_{\substack{xz = m\\x\ {\rm odd}\\0\leq y <z}} f\left(\frac{x\tau +y}{z}\right),&\\
& {\bf U}_m f(\tau) = \sum_{\substack{xz = m\\z\ {\rm odd}\\0\leq y <z}} f\left(\frac{x\tau +y}{z}\right),\qquad &&
 {\bf O}_m f(\tau) = \sum_{\substack{xz = m\\x,z\ {\rm even},\ y\ {\rm odd}\\0\leq y <z}} f\left(\frac{x\tau +y}{z}\right).&
\end{alignat*}
Then we have
\begin{gather}\label{tildeTFormula}
 {\bf \tilde{T}}_m=
 \begin{cases}
 {\bf R}_m +{\bf R}_{2m} -{\bf U}_{2m}-{\bf D}_{2m},& m\ \hbox{\rm even},\\
 {\bf R}_m,& m\ \hbox{\rm odd},
 \end{cases}
 \end{gather}
 since when $m$ is even the second term of ${\bf \tilde{T}}_m$ excludes the cases when $x$ is odd and when $z$ is odd which are disjoint.

We also have
\begin{gather}\label{OEven}
 {\bf O}_{4m} = {\bf R}_{4m} -{\bf R}_m -{\bf U}_{4m} - {\bf D}_{4m},
\end{gather}
since the cases when $x$, $y$ and $z$ are all even, when $x$ is odd and when $z$ is odd must be excluded from the sum in ${\bf R}_{4m}$ and these are disjoint.

If $m$ is odd then
\begin{gather}\label{OOdd}
 {\bf O}_{2m} = 0 = {\bf R}_{2m} - {\bf U}_{2m} - {\bf D}_{2m},
\end{gather}
since in this case either $x$ or $z$ is odd (but not both).

For ${\bf T}_m$ we have the following
\begin{gather*}
 \sum_{\substack{\divides{r^2}{m}\\r\ {\rm odd}}} {\bf T}_{m/r^2} =
 \begin{cases}
 {\bf O}_{2m}+{\bf U}_m + {\bf D}_m, & m\ {\rm even},\\
 {\bf R}_m, & m\ {\rm odd}.
 \end{cases}
 \end{gather*}

 This follows from the terms in $M^m$ in each case and the fact that the condition that the $gcd$ of the entries is 1 for each element of~$M^m$. In particular, this establishes the required result if~$m$ is odd.

 We also have
 \begin{gather}\label{tildeT2mFormula}
 {\bf \tilde{T}}_{2m}-{\bf \tilde{T}}_m = {\bf O}_{4m} + {\bf U}_{2m}+ {\bf D}_{2m}.
 \end{gather}

 This follows from (\ref{tildeTFormula}) as follows. If $m$ is even then
 \begin{gather*}
 {\bf \tilde{T}}_{2m} - {\bf \tilde{T}}_m = {\bf R}_{2m}+{\bf R}_{4m} - {\bf U}_{4m} - {\bf D}_{4m} -{\bf R}_{m}-{\bf R}_{2m} + {\bf U}_{2m} + {\bf D}_{2m} \\
 \hphantom{{\bf \tilde{T}}_{2m} - {\bf \tilde{T}}_m}{} = {\bf R}_{4m} -{\bf R}_m - {\bf U}_{4m} - {\bf D}_{4m} + {\bf U}_{2m} + {\bf D}_{2m} ={\bf O}_{4m} +{\bf U}_{2m} + {\bf D}_{2m}
 \end{gather*}
 using (\ref{OEven}). While if $m$ is odd then
\begin{gather*}
 {\bf \tilde{T}}_{2m} - {\bf \tilde{T}}_m = {\bf R}_{2m}+{\bf R}_{4m} - {\bf U}_{4m} - {\bf D}_{4m} -{\bf R}_m\\
 \hphantom{{\bf \tilde{T}}_{2m} - {\bf \tilde{T}}_m}{} = {\bf R}_{2m} +{\bf O}_{4m} ={\bf O}_{4m} +{\bf U}_{2m} + {\bf D}_{2m}
 \end{gather*}
 using (\ref{OOdd}).

 To establish the proposition if $m$ is even, we have from (\ref{tildeT2mFormula})
 \begin{gather*}
 {\bf \tilde{T}}_m = {\bf O}_{2m} +{\bf U}_m + {\bf D}_m+{\bf \tilde{T}}_{m/2} = \sum_{\divides{r^2}{m}} {\bf T}_{m/r^2} + {\bf \tilde{T}}_{m/2}.
 \end{gather*}

 The result now follows by induction on $\alpha$ where $m = 2^\alpha \beta$ with $\beta$ odd.
\end{proof}

\subsection[$(2{+})$-replicability]{$\boldsymbol{(2{+})}$-replicability}\label{2Arepl}

The operators ${\bf T}_m$ and ${\bf \tilde{T}}_m$ just defined map the field of modular functions for $\Gp$ to itself. If~$T_{2A}$ is the normalized Hauptmodul of $\Gp$ this means that ${\bf T}_m(T_{2A})(z)$ is a rational function of~$T_{2A}(z)$ and since ${\bf T}_m(T_{2A})(z)$ has no poles in the upper half-plane this rational function is actually a polynomial. From the power series expansion we can see that it has to be the $m$-th Faber polynomial of $T_{2A}$. We have just said that
\begin{gather}\label{2Aself}
 P_{m}(T_{2A}(\tau))= \sum_{\substack{ad=m\\0\leq b<d}} T_{2A}\left(\frac{a\tau+b}{d} \right)+ \sum_{\substack{ad=m \\ d \ \text{even} \\ 0\leq b<d}} T_{2A}\left(\frac{2a\tau+b}{d}\right).
\end{gather}
As discussed in the introduction, we will show that the situation is entirely analogous to the situation with ${\rm PSL}(2,\intn)$, the monster and ``ordinary'' replication. For $\Gp$ and the baby monster we will show
that the appropriate generalization is the following:
\begin{Definition}
A function $f$ is \tp-replicable if there are $f^{[n]}$ and $f^{[n\sqrt{2}]}$, for $n\in \mathbb{N}$, such that
\begin{gather}\label{def2A}
P_{n,f}(f)= \sum_{\substack{ad=n\\0\leq b<d}} f^{\left[a\right]}\left(\frac{a\tau+b}{d} \right)+ \sum_{\substack{ad=n \\ d \ \text{even} \\ 0\leq b<d}} f^{\left[a\sqrt{2}\right]}\left(\frac{2a\tau+b}{d}\right).
\end{gather}
\end{Definition}

Equation (\ref{2Aself}) is the \tp-self-replication property of $T_{2A}$.
We make a few remarks on this definition.

\begin{Remark} Given a \tp-replicable function $f$, its \tp-replicates are not determined uniquely. For example, for $m=2$, equation (\ref{def2A}) becomes
\begin{gather*}
f^{[2]}(2\tau)+f^{[\sqrt{2}]}(\tau)+f^{[\sqrt{2}]}\left(\tau+\frac{1}{2}\right)+f\left(\frac{\tau}{2}\right)+f\left(\frac{\tau+1}{2}\right)=P_{2,f}(f),
\end{gather*}
and we can see that $f^{[2]}$ is known when $f$ and $f^{[\sqrt{2}]}$ are known. Also, the odd-power coefficients of the replicates $f^{[\sqrt{2}n]}$ can be changed freely and identity (\ref{def2A}) is still true. We will see instances of \tp-replicable functions that are \tp-replicable in different ways, i.e., have different \tp-replicates. For example, in Table~\ref{Table3} the trace functions for the $2\cdot\mathbb{B}$ classes~$2e$ and $4d$ are both $T_{4C}$, the Hauptmodul for $\Gamma_0(4)$ which corresponds to the class $4C$ in the monster. However, for the classes $2e$ and $4d$, $f^{[\sqrt{2}]}$ is given by the monstrous functions $T_{2B}$ and $T_{4C}$ respectively. In other words there are two \tp-replication structures which are compatible with the Hauptmodul $T_{4C}$ and both of these occur for the baby monster. One might think that for ``ordinary'' replication there is in principle a similar lack of uniqueness in the definition of replicates. For example, the monstrous classes $27A$ and $27B$ have the same trace functions (in fact they are the only rational classes with this property). But both cube to~$9B$. More generally, as pointed out by Norton~\cite{N2}, the ``ordinary'' replicates of a replicable function are unique. So for ``ordinary'' replication we have a unique ``power map'' structure.
\end{Remark}

\begin{Remark}\label{equivrep2Arep}
If a function $f$ is \tp-replicable then it is replicable with
\begin{gather*}
f^{(n)}(z)= \begin{cases}
 f^{[n]}(z), & n \ \mbox{odd}, \\
 \displaystyle f^{[n]}(z)+f^{[\frac{n}{\sqrt{2}}]}\left(\frac{z}{2}\right)+f^{[\frac{n}{\sqrt{2}}]}\left(\frac{z+1}{2}\right), & n \ \mbox{even}. \\
 \end{cases}
\end{gather*}

Also, we can see that if $f$ is replicable and if, for every $n$ even, we can write
\begin{gather*}
f^{(n)}(z)=f^{[n]}(z)+f^{[\frac{n}{\sqrt{2}}]}\left(\frac{z}{2}\right)+f^{[\frac{n}{\sqrt{2}}]}\left(\frac{z+1}{2}\right),
\end{gather*}
for some $f^{[n]}$, $f^{[\frac{n}{\sqrt{2}}]}$, then $f$ is \tp-replicable. For $n$ odd we obviously have $f^{[n]}=f^{(n)}$. This is shown by the following simple manipulation
\begin{gather*}
 \sum_{\substack{ad=n \\ 0\leq b < d}} f^{(a)}\left(\frac{az+b}{d}\right)
 = \sum_{\substack{ad=n \\ a \ \text{odd} \\0\leq b < d}} f^{(a)}\left(\frac{az+b}{d}\right)+\sum_{\substack{ad=n \\ a \ \text{even} \\0\leq b < d}} f^{(a)}\left(\frac{az+b}{d}\right) \\
\hphantom{\sum_{\substack{ad=n \\ 0\leq b < d}} f^{(a)}\left(\frac{az+b}{d}\right)}{}
=\sum_{\substack{ad=n \\ a \ \text{odd} \\0\leq b < d}} f^{[a]}\left(\frac{az+b}{d}\right)+ \sum_{\substack{ad=n \\ a \ \text{even} \\0\leq b < d}} f^{[a]}\left(\frac{az+b}{d}\right) \\
\hphantom{\sum_{\substack{ad=n \\ 0\leq b < d}} f^{(a)}\left(\frac{az+b}{d}\right)=}{} +\sum_{\substack{ad=n \\ a \ \text{even} \\0\leq b < d}} f^{[\sqrt{2}\frac{a}{2}]}\left(\frac{az+b}{2d}\right)+\sum_{\substack{ad=n \\ a \ \text{even} \\0\leq b < d}} f^{[\sqrt{2}\frac{a}{2}]}\left(\frac{az+b+d}{2d}\right) \\
\hphantom{\sum_{\substack{ad=n \\ 0\leq b < d}} f^{(a)}\left(\frac{az+b}{d}\right)}{}
 =\sum_{\substack{ad=n \\0\leq b < d}} f^{[a]}\left(\frac{az+b}{d}\right)+\sum_{\substack{ad=n \\ a \ \text{even} \\0\leq b < 2d}} f^{[\sqrt{2}\frac{a}{2}]}\left(\frac{2\left(\frac{a}{2}\right)z+b}{2d}\right)\\
\hphantom{\sum_{\substack{ad=n \\ 0\leq b < d}} f^{(a)}\left(\frac{az+b}{d}\right)}{}
 =\sum_{\substack{ad=n \\0\leq b < d}} f^{[a]}\left(\frac{az+b}{d}\right)+\sum_{\substack{ad=n \\ d \ \text{even} \\0\leq b < d}} f^{[\sqrt{2}a]}\left(\frac{2az+b}{d}\right).
\end{gather*}

Finally, if $f$ is replicable then $f$ is always \tp-replicable by taking, for example, $f^{[n\sqrt{2}]}=0$.

Note that this is in contrast with ordinary replication for which the replication equations uniquely fix the replicates.
\end{Remark}

\begin{Remark} Because of Remark~\ref{equivrep2Arep} what we are actually interested in is finding the possible \tp-replicables for a given \tp-replicable function. For example, we could ask if a \tp-replicable function is completely \tp-replicable in the sense given below. There are examples of Hauptmoduls that are not complete replicable functions but are completely \tp-replicable. As we will see, some examples of these are $T_{4\sim b}$, $T_{12\sim d}$ among many others (see~\cite{N3} for the notation).
\end{Remark}

For the monster, since replicates corresponds to power maps, the replicable functions sa\-tisfy a stronger property called complete replicability. Namely that the replicate functions are themselves replicable and that
taking replicates is ``commutative''. We make a corresponding definition
for \tp-replicability.
\begin{Definition}
We say that a function $f$ is completely \tp-replicable if
\begin{itemize}\itemsep=0pt
\item[--] it is \tp-replicable, with replicates $f^{[n]}$, $n\in \nat \cup\sqrt{2}\nat$, and
\item[--] for every $n\in \nat\cup \sqrt{2}\nat$ the function $f^{[n]}$ is \tp-replicable with \tp-replicates $\big(f^{[n]}\big)^{[m]}=f^{[mn]}$, for any $m\in \nat \cup\sqrt{2}\nat$.
\end{itemize}
\end{Definition}

In the following section we show that certain Hauptmoduls are completely \tp-replicable. The list will include almost all the Hauptmoduls which occur in moonshine for the baby monster. The proof of \tp-replicability for the remaining baby monstrous moonshine functions will be given at the end of Section~\ref{BMandrep}.

\section[Complete $(2+)$-replicability and Hauptmoduls]{Complete $\boldsymbol{(2+)}$-replicability and Hauptmoduls}\label{CRHauptmoduls}

The following result is Theorem~5.15 in Ferenbaugh's Ph.D.\ Thesis \cite{Fer} and will be useful in proving that certain Hauptmoduls that are not completely replicable are completely \tp-replicable.
We refer to \cite{CN} and \cite{FMN} for the notation on the Atkin--Lehner involutions $W_e$ and on groups of the form $n\,|\, h+e_1,e_2,\ldots$ and $n\parallel h+e_1,e_2,\ldots$.

\begin{Theorem}\label{Feren}
Let $p$ be prime and $f$ be the Hauptmodul for some group $np\,|\, h+e_1,e_2\ldots$ with $p\nmid h$. Then
\begin{itemize}\itemsep=0pt
\item[--] if $p$ is one of the $e_1$, $e_2,\ldots$ then $f^{(p)}(z)=f(z)+f\big(\frac{z}{p}\big)+\dots+f\big(\frac{z+p-1}{p}\big)$,
\item[--] if $p$ is not one of the $e_1, e_2,\ldots$ but $ep$ is, for some $e$ with $p\!\nmid \!e$, then $f^{(p)}(W_ez)=f(W_pz)+f\big(\frac{z}{p}\big)+\dots+f\big(\frac{z+p-1}{p}\big)$,
\item[--] if $ep$ is one of the $e_1$, $e_2,\ldots$ with $p\,|\, e$ then $f^{(p)}(W_ez)=f\big(\frac{z}{p}\big)+\dots+f\big(\frac{z+p-1}{p}\big)$.
\end{itemize}
\end{Theorem}

Remark \ref{equivrep2Arep} gives motivation to find formulas of type $A(z)=B(z)+C\big(\frac{z}{2}\big)+C\big(\frac{z+1}{2}\big)$. This will help us finding \tp-replicates for certain Hauptmoduls.
Before we proceed, we need some notation in order to separate even Atkin--Lehner involutions from the odd ones.

From now on, any group $n\,|\, h+e,f,\ldots$ will be written as $n\,|\, h+O_1+2^kO_2$ where $2^k$ is the highest power of $2$ that divides $\frac{n}{h}$, the set $O_1$ is the set of odd Atkin--Lehner involutions of the group and $O_2$ is the set of even Atkin--Lehner involutions of the group divided by $2^k$. We always include $1$ in $O_1$.

For example, the group $30+6$, $10$, $15$ will be written as $30+O_1+2O_2$, with $O_1= \{1,15 \}$ and $O_2=\left\{3,5\right\}$.

We should also note that the elements of both $O_1$ and $O_2$ are odd, and these sets can satisfy $O_1=O_2$ as, for example, in $6+1$, $2$, $3$, $6$, where $O_1=O_2=\{1,3\}$.

Recall, $G^\alpha$ denotes $
\left[\begin{smallmatrix}
1 & \alpha \\
0 & 1 \\
\end{smallmatrix}\right]
G
\left[\begin{smallmatrix}
1 & \alpha \\
0 & 1 \\
\end{smallmatrix}\right]^{-1}$.

\begin{Remark}\label{conj1/2}
If the matrix
$\left[\begin{smallmatrix}
1 & \frac{1}{2} \\
0 & 1 \\
\end{smallmatrix}\right]$
does not normalize the genus zero group $2^kN\,|\,h+O_1+2^kO_2$, with $N$ odd, then
\begin{gather*}
T_{(2^kN|h+O_1+2^kO_2)^{\frac{1}{2}}}(z)=-T_{2^kN|h+O_1+2^kO_2}\left(z+\frac{1}{2}\right).
\end{gather*}

If it normalizes, we have
\begin{gather*}
T_{2^kN|h+O_1+2^kO_2}\left(z+\frac{1}{2}\right)=-T_{2^kN|h+O_1+2^kO_2}(z),
\end{gather*}
and this happens exactly when either $h$ is even or $h=1$, $k\geq 2$ and $O_2$ is empty.

Also, this matrix conjugates the groups $2N+O_1$ and $4N+O_1+4O_1$ onto each other, when $(2,N)=1$.
\end{Remark}

We make the following observation which will be used several times in what follows.
\begin{Remark}
For a Hauptmodul $f$ of a genus zero group $G$, $\alpha=\frac{1}{2},\frac{1}{4},\frac{1}{8}$ and $n$ odd we have
\begin{gather*}
\big(f^{\alpha}\big)^{(n)}=\big(f^{(n)}\big)^{\alpha}.
\end{gather*}

In particular, $\big(\big(f^{\alpha}\big)^{(m)}\big)^{(n)}=\big(f^{\alpha}\big)^{(mn)}$, for $m$, $n$ odd.
\end{Remark}

\begin{Lemma}\label{decomp}
We have the following identities:
\begin{enumerate}\itemsep=0pt
 \item[$1.$] If $(2,N) = 1$ then
 \begin{gather*}
 T_{N+O_1}(z)=T_{2N+O_1+2O_1}(z)+T_{2N+O_1+2O_1}\left(\frac{z}{2}\right)+T_{2N+O_1+2O_1}\left(\frac{z+1}{2}\right)
 \end{gather*}
 or, equivalently,
 \begin{gather*}
T_{2N+O_1+2O_1}(z)=T_{N+O_1}(z)+T_{(2N+O_1+2O_1)^\frac{1}{2}}\left(\frac{z}{2}\right)+ T_{(2N+O_1+2O_1)^\frac{1}{2}}\left(\frac{z+1}{2}\right).
 \end{gather*}
\item[$2.$]
If $(2,N)=1$, $k\geq 1$ and
 $O_2$ is not empty then
\begin{gather*}
T_{2^kN+O_1}(z)=T_{2^kN+O_1+2^kO_2}(z)+T_{(2^{k+1}N+O_1+2^{k+1}O_2)^\frac{1}{2}}\left(\frac{z}{2}\right)\\
\hphantom{T_{2^kN+O_1}(z)=}{} +T_{(2^{k+1}N+O_1+2^{k+1}O_2)^\frac{1}{2}}\left(\frac{z+1}{2}\right).
\end{gather*}

In particular, if $k=1$ and $O_1=O_2$ we have
 \begin{gather*}
T_{2N+O_1}(z)=T_{2N+O_1+2O_1}(z)+T_{2N+O_1}\left(\frac{z}{2}\right)+T_{2N+O_1}\left(\frac{z+1}{2}\right).
 \end{gather*}

\item[$3.$]
 If $(2,N)=1$, $k\geq 1$ and $O_2$ is not empty then
\begin{gather*}
T_{2^k N+O_1+2^kO_2}(z)=T_{2^k N+O_1}(z)+T_{2^{k+1}N+O_1+2^{k+1}O_2}\left(\frac{z}{2}\right)\\
\hphantom{T_{2^k N+O_1+2^kO_2}(z)=}{} +T_{2^{k+1}N+O_1+2^{k+1}O_2}\left(\frac{z+1}{2}\right).
\end{gather*}
\end{enumerate}
\end{Lemma}
\begin{proof} The first identity is Theorem \ref{Feren} applied to $f=T_{2N+O_1+2O_1}$ and $p=2$. The third identity is also a consequence of Theorem \ref{Feren} applied to $f=T_{2^{k+1}N+O_1+2^{k+1}O_2}$ and $p=2$. Since $2^k\parallel N$, for $k\geq 1$, we have
\begin{gather*}T_{2^{k+1}N+O_1+2^{k+1}O_2}\left(\frac{z}{2}\right)+T_{2^{k+1}N+O_1+2^{k+1}O_2}\left(\frac{z+1}{2}\right)=T_{2^kN+O_1}\left(W_{e'}z\right),\end{gather*}
where $e'=2^ke$ for some $e\in O_2$.
But \begin{gather*}T_{2^kN+O_1}(z)+T_{2^kN+O_1}(W_{e'}z)=T_{2^kN+O_1+2^kO_2}(z)\end{gather*} and this proves the third identity. The second identity is the third one written in a different way.
\end{proof}

\begin{table}[t!]
\centering
\scalebox{0.725}{
\begin{tabular}{|c||c|c|c|c|}
\hline
$g$ & $2N+O_1+2O_2$ &$4N+O_1+4O_2$ & $4N+O_1+4O_2$ & $8N+O_1+8O_2$ \\ \hline
$g^{[\sqrt{2}]}$ & $2N+O_1+2O_1$ &$2N+O_1$ & $4N+O_1$ & $4N+O_1$ \\ \hline
$g^{[2]}$ & $2N+O_1+2O_1$ &$2N+O_1+2O_1$ & $2N+O_1$ & $4N+O_1$ \\ \hline
$g^{[2\sqrt{2}]}$ & $ \cdots$ &$2N+O_1+2O_1$ & $2N+O_1+2O_1$ & $2N+O_1$ \\ \hline
$g^{[4]}$ & & $\cdots$ & $2N+O_1+2O_1$ & $2N+O_1+2O_1$ \\ \hline
$g^{[4\sqrt{2}]}$ & & & $\cdots$ & $2N+O_1+2O_1$ \\ \hline
$\vdots$ & & & & $\cdots$ \\ \hline
\hline
$g$ & $8N+O_1+8O_2$ & $16N+O_1+16O_2$ & $16N+O_1+16O_2$ &$32N+O_1+32O_2$ \\ \hline
$g^{[\sqrt{2}]}$ & $8N+O_1$ & $8N+O_1$ & $16N+O_1$ &$16N+O_1 $ \\ \hline
$g^{[2]}$ & $4N+O_1$ & $8N+O_1$ & $8N+O_1$ &$16N+O_1 $ \\ \hline
$g^{[2\sqrt{2}]}$ & $4N+O_1$ & $4N+O_1$ &$8N+O_1$ &$8N+O_1 $ \\ \hline
$g^{[4]}$ & $2N+O_1$ & $4N+O_1$ &$4N+O_1$ &$8N+O_1 $ \\ \hline
$g^{[4\sqrt{2}]}$ & $2N+O_1+2O_1$ & $2N+O_1$ &$4N+O_1$ &$4N+O_1 $ \\ \hline
$g^{[8]}$ & $2N+O_1+2O_1$ & $2N+O_1+2O_1$ &$2N+O_1$ &$4N+O_1 $ \\ \hline
$g^{[8\sqrt{2}]}$ & $\cdots$ & $2N+O_1+2O_1$ &$2N+O_1+2O_1$ &$2N+O_1 $ \\ \hline
$g^{[16]}$ & & $\cdots$ &$2N+O_1+2O_1$ &$2N+O_1+2O_1$ \\ \hline
$g^{[16\sqrt{2}]}$ & & &$\cdots$ &$2N+O_1+2O_1$ \\ \hline
$\vdots$ & & & &$\cdots$ \\ \hline
\hline
$g$ &$2N|2+O_1$ &$4N|2+O_1+2O_2 $ & $8N|2+O_1+4O_2$ &$8N|2+O_1+4O_2$ \\ \hline
$g^{[\sqrt{2}]}$ &$2N+O_1+2O_1$ &$4N+O_1+4O_2 $ & $4N|2+O_1+2O_2$ &$8N+O_1+8O_2$ \\ \hline
$g^{[2]}$ &$2N+O_1+2O_1$ &$2N+O_1$ & $4N+O_1+4O_2$ &$4N+O_1$ \\ \hline
$g^{[2\sqrt{2}]}$ &$\cdots$ &$2N+O_1+2O_1$ & $2N+O_1$ &$4N+O_1$ \\ \hline
$g^{[4]}$ & &$2N+O_1+2O_1$ & $2N+O_1+2O_1$ &$2N+O_1$ \\ \hline
$g^{[4\sqrt{2}]}$ & &$\cdots$ & $2N+O_1+2O_1$ &$2N+O_1+2O_1$ \\ \hline
$g^{[8]}$ & & & $\cdots$ &$2N+O_1+2O_1$ \\ \hline
$\vdots$ & & & & $\cdots$ \\ \hline
\hline
$g$ &$16N|2+O_1+8O_2$ &$16N|2+O_1+8O_2$ &$32N|2+O_1+16O_2$ & $32N|2+O_1+16O_2$ \\ \hline
$g^{[\sqrt{2}]}$ &$8N|2+O_1+4O_2$ &$16N+O_1+16O_2$ &$16N|2+O_1+8O_2$ & $32N+O_1+32O_2$ \\ \hline
$g^{[2]}$ &$8N+O_1+8O_2$ &$8N+O_1$ &$16N+O_1+16O_2$ & $16N+O_1$ \\ \hline
$g^{[2\sqrt{2}]}$ &$4N+O_1$ &$8N+O_1$ &$8N+O_1$ & $16N+O_1$ \\ \hline
$g^{[4]}$ &$4N+O_1$ &$4N+O_1$ &$8N+O_1$ & $8N+O_1$ \\ \hline
$g^{[4\sqrt{2}]}$ &$2N+O_1$ &$4N+O_1$ &$4N+O_1$ & $8N+O_1$ \\ \hline
$g^{[8]}$ &$2N+O_1+2O_1$ &$2N+O_1$ &$4N+O_1$ & $4N+O_1$ \\ \hline
$g^{[8\sqrt{2}]}$ &$2N+O_1+2O_1$ &$2N+O_1+2O_1$ &$2N+O_1$ & $4N+O_1$ \\ \hline
$g^{[16]}$ &$\cdots$ &$2N+O_1+2O_1$ &$2N+O_1+2O_1$ & $2N+O_1$ \\ \hline
$g^{[16\sqrt{2}]}$ & &$\cdots$ &$2N+O_1+2O_1$ & $2N+O_1+2O_1$ \\ \hline
$g^{[32]}$ & & &$\cdots$ & $2N+O_1+2O_1$ \\ \hline
$\vdots$ & & & & $\cdots$ \\ \hline
\hline
$g$ & $8N|4+O_1+2O_2$ & $16N|4+O_1+4O_2$ & $16N|4+O_1+4O_2$ & $32N|4+O_1+8O_2$ \\ \hline
$g^{[\sqrt{2}]}$ & $8N|2+O_1+4O_2$ & $16N|2+O_1+8O_2$ & $8N|4+O_1+2O_2$ & $32N|2+O_1+16O_2$ \\ \hline
$g^{[2]}$ & $4N|2+O_1+2O_2$ & $8N|2+O_1+4O_2$ & $8N|2+O_1+4O_2$ & $16N|2+O_1+8O_2$ \\ \hline
$g^{[2\sqrt{2}]}$ & $4N+O_1+4O_2$ & $8N+O_1+8O_2$ & $4N|2+O_1+2O_2$ & $16N+O_1+16O_2$ \\ \hline
$g^{[4]}$ & $2N+O_1$ & $4N+O_1$ & $4N+O_1+4O_2$ & $8N+O_1$ \\ \hline
$g^{[4\sqrt{2}]}$ & $2N+O_1+2O_1$ & $4N+O_1$ & $2N+O_1$ & $8N+O_1$ \\ \hline
$g^{[8]}$ & $2N+O_1+2O_1$ & $2N+O_1$ & $2N+O_1+2O_1$ & $4N+O_1$ \\ \hline
$g^{[8\sqrt{2}]}$ & $\cdots$ & $2N+O_1+2O_1$ & $2N+O_1+2O_1$ & $4N+O_1$ \\ \hline
$g^{[16]}$ & & $2N+O_1+2O_1$ & $\cdots$ & $2N+O_1$ \\ \hline
$g^{[16\sqrt{2}]}$ & & $\cdots$ & & $2N+O_1+2O_1$ \\ \hline
$g^{[32]}$ & & & & $2N+O_1+2O_1$ \\ \hline
$\vdots$ & & & & $\cdots$ \\ \hline
\hline
$g$ & $4N+O_1$ & $4N|2+O_1+2O_1$ & $4N|2+O_1$ & $(4N+O_1+4O_2)^{\frac{1}{2}}$ \\ \hline
$g^{[\sqrt{2}]}$ & $(4N+O_1+4O_2)^\frac{1}{2}$ & $(2N+O_1+2O_1)^\frac{1}{2}$ & $(4N+O_1+4O_2)^\frac{1}{2}$ & $2N+O_1+2O_2$ \\ \hline
$g^{[2]}$ & $2N+O_1+2O_2$ & $N+O_1$ & $2N+O_1+2O_2$ & $2N+O_1+2O_1$ \\ \hline
$g^{[2\sqrt{2}]}$ & $2N+O_1+2O_1$ & $2N+O_1+2O_1$ & $2N+O_1+2O_1$ & $2N+O_1+2O_1$ \\ \hline
$g^{[4]}$ & $2N+O_1+2O_1$ & $2N+O_1+2O_1$ & $2N+O_1+2O_1$ & $\cdots$ \\ \hline
$\vdots$ & $\cdots$ & $\cdots$ & $\cdots$ & \\ \hline
\end{tabular}}
\vspace{-1.4mm}

\caption{\tp-replicates of Hauptmoduls.}\label{Table1}
\end{table}

\begin{table}[t!]\centering
\scalebox{0.725}{
\begin{tabular}{|c||c|c|c|c|}
\hline
$g$ & $8N|2+O_1$ & $8N|4+O_1+2O_1$ & $8N|2+O_1$ & $8N|4+O_1$ \\ \hline
$g^{[\sqrt{2}]}$ & $(8N+O_1+8O_2)^\frac{1}{2}$ & $(4N|2+O_1+2O_1)^\frac{1}{4}$ & $(8N|2+O_1+4O_1)^\frac{1}{4}$ & $(8N|2+O_1+4O_1)^\frac{1}{4}$ \\ \hline
$g^{[2]}$ & $4N+O_1+4O_2$ & $4N|2+O_1+2O_1$ & $4N+O_1$ & $4N|2+O_1$ \\ \hline
$g^{[2\sqrt{2}]}$ & $4N+O_1$ & $(2N+O_1+2O_1)^\frac{1}{2}$ & $(4N+O_1+4O_1)^\frac{1}{2}$ & $(4N+O_1+4O_1)^\frac{1}{2}$ \\ \hline
$g^{[4]}$ & $2N+O_1 $ & $N+O_1 $ & $2N+O_1+2O_1$ & $2N+O_1+2O_1$ \\ \hline
$g^{[4\sqrt{2}]}$ & $2N+O_1+2O_1$ & $2N+O_1+2O_1$ & $2N+O_1+2O_1$ & $2N+O_1+2O_1$ \\ \hline
$g^{[8]}$ & $2N+O_1+2O_1$ & $2N+O_1+2O_1$ & $2N+O_1+2O_1$ & $2N+O_1+2O_1$ \\ \hline
$\vdots$ & $\cdots$ & $\cdots$ & $\cdots$ & $\cdots$ \\ \hline
\hline
$g$ & $(2N+O_1+2O_1)^{\frac{1}{2}}$ & $N+O_1$ & $(8N+O_1+8O_2)^\frac{1}{2}$ & $(8N+O_1+8O_2)^{\frac{1}{2}}$ \\ \hline
$g^{[\sqrt{2}]}$ & $N+O_1 $ & $2N+O_1+2O_1$ & $4N+O_1+4O_2$ & $8N+O_1+8O_2$ \\ \hline
$g^{[2]}$ & $2N+O_1+2O_1$ & $2N+O_1+2O_1$ & $4N+O_1$ & $4N+O_1$ \\ \hline
$g^{[2\sqrt{2}]}$ & $2N+O_1+2O_1$ & $\cdots$ & $2N+O_1$ & $2N+O_1$ \\ \hline
$g^{[4]}$ & $\cdots$ & & $2N+O_1+2O_1$ & $2N+O_1+2O_1$ \\ \hline
$g^{[4\sqrt{2}]}$ & & & $2N+O_1+2O_1$ & $2N+O_1+2O_1$ \\ \hline
$\vdots$ & & & $\cdots$ & $\cdots$ \\ \hline
\hline
$g$ & $(16N+O_1+16O_2)^{\frac{1}{2}}$ & $(32N+O_1+32O_2)^{\frac{1}{2}}$ & $(4N|2+O_1+2O_1)^{\frac{1}{4}}$ & $(8N|2+O_1+4O_2)^{\frac{1}{4}}$ \\ \hline
$g^{[\sqrt{2}]}$ & $8N+O_1+8O_2$ & $16N+O_1+16O_2$ & $4N|2+O_1+2O_1$ &$4N|2+O_1$ \\ \hline
$g^{[2]}$ & $8N+O_1$ & $16N+O_1$ & $(2N+O_1+2O_1)^{\frac{1}{2}}$ &$(4N+O_1+4O_2)^{\frac{1}{2}}$ \\ \hline
$g^{[2\sqrt{2}]}$ & $4N+O_1$ & $8N+O_1$ & $N+O_1$ &$2N+O_1+2O_2$ \\ \hline
$g^{[4]}$ & $4N+O_1$ & $8N+O_1$ & $2N+O_1+2O_1$ &$2N+O_1+2O_1$ \\ \hline
$g^{[4\sqrt{2}]}$ & $2N+O_1$ & $4N+O_1$ & $2N+O_1+2O_1$ &$2N+O_1+2O_1$ \\ \hline
$g^{[8]}$ & $2N+O_1+2O_1$ & $4N+O_1$ & $\cdots$ &$\cdots$ \\ \hline
$g^{[8\sqrt{2}]}$ & $2N+O_1+2O_1$ & $2N+O_1$ & & \\ \hline
$g^{[16]}$ & $\cdots$ & $2N+O_1+2O_1$ & & \\ \hline
$g^{[16\sqrt{2}]}$ & & $2N+O_1+2O_1$ & & \\ \hline
$\vdots$ & & $\cdots$ & & \\ \hline
\hline
$g$ & $(16N|2+O_1+8O_1)^{\frac{1}{4}}$ & $(8N|4+O_1+2O_1)^{\frac{1}{8}}$ & $(8N|4+O_1+2O_1)^{\frac{1}{8}}$ \\ \cline{1-4}
$g^{[\sqrt{2}]}$ & $8N|2+O_1$ & $8N|4+O_1+2O_1$ & $8N|2+O_1+2O_1$ \\ \cline{1-4}
$g^{[2]}$ & $(8N+O_1+8O_1)^{\frac{1}{2}}$ & $(4N|2+O_1+2O_1)^{\frac{1}{4}}$ & $(4N|2+O_1+2O_1)^{\frac{1}{4}}$ \\ \cline{1-4}
$g^{[2\sqrt{2}]}$ & $4N+O_1+4O_1$ & $4N|2+O_1+2O_1$ & $4N|2+O_1+2O_1$ \\ \cline{1-4}
$g^{[4]}$ & $4N+O_1$ & $(2N+O_1+2O_1)^{\frac{1}{2}}$ & $(2N+O_1+2O_1)^{\frac{1}{2}}$ \\ \cline{1-4}
$g^{[4\sqrt{2}]}$ & $2N+O_1$ & $N+O_1$ & $N+O_1$ \\ \cline{1-4}
$g^{[8]}$ & $2N+O_1+2O_1$ & $2N+O_1+2O_1$ & $2N+O_1+2O_1$ \\ \cline{1-4}
$g^{[8\sqrt{2}]}$ & $2N+O_1+2O_1$ & $2N+O_1+2O_1$ & $2N+O_1+2O_1$ \\ \cline{1-4}
$g^{[16]}$ & $\cdots$ & $\cdots$ & $\cdots$ \\ \cline{1-4}
\end{tabular}}

\caption{\tp-replicates of Hauptmoduls (continued).}\label{Table2}
\end{table}

\begin{Theorem}\label{repl3AHaupt}
Every function $g$ of type as given in Tables~{\rm \ref{Table1}} and~{\rm \ref{Table2}} is \tp-replicable. The $\big[2^{\frac{n}{2}}\big]$-replicates, for $n$ a non-negative integer, are the ones given in those tables and replicates $g^{[2^{\frac{n}{2}}m]}$, for~$n$ a non-negative integer and $m$ a positive odd integer, are given by $g^{[2^{\frac{n}{2}}m]}=\big(g^{[2^{\frac{n}{2}}]}\big)^{(m)}$.
\end{Theorem}

\begin{proof} We start by proving the result for completely replicable functions and then use this to complete the proof for conjugates of completely replicable functions. For the cases where the function is completely replicable we know how replication works and for each $g^{(2^m)}$ we will either apply Lemma~\ref{decomp} or, in view of Remark~\ref{equivrep2Arep}, take $g^{(2^m)}$ as $g^{[2^m]}$ and any function satisfying the condition in Remark~\ref{conj1/2} as $g^{[\frac{2^m}{\sqrt{2}}]}$. This defines the $(2+)$-replicates $g^{[m]}$, with~$m$ a power of~$\sqrt{2}$. Having done so, and because of Remark~\ref{equivrep2Arep} again, it is enough in order to finish the proof in this case to show that, for all positive $m$ and odd $n$, we have
\begin{gather}\label{aaa}
 g^{(2^mn)}(z)=\left(g^{[2^m]}\right)^{(n)}(z)+\left(g^{[\frac{2^m}{\sqrt{2}}]}\right)^{(n)}\left(\frac{z}{2}\right)+\left(g^{[\frac{2^m}{\sqrt{2}}]}\right)^{(n)}\left(\frac{z+1}{2}\right).
\end{gather}

 This is a case by case check on every $g^{(2^m)}$.

1.~Our first case is when we are using some identity from Lemma \ref{decomp} to decompose $g^{(2^m)}$. In this case, we must have $g^{(2^m)}=2^kN+O_1+2^kO_2$, i.e., $h=1$, and $g^{(2^mn)}=\left(g^{(2^m)}\right)^{(n)}=2^kN'+O'_1+2^kO'_2$ where $N'=N/\gcd(N,n)$, $O'_1=\{e\in O_1 \,|\, e \mbox{ divides } N'\}$ and $O'_2=\{e\in O_2 \,|\, e \mbox{ divides } N'\}$. We fix an odd $n$ and divide this case further into three different subcases.

i)~If $O_2=\varnothing$ (which implies $O'_2=\varnothing$) then the possibilities, depending on $k$, for $g^{[2^m]}$ and $g^{[\frac{2^m}{\sqrt{2}}]}$ when we apply Lemma \ref{decomp} to decompose $g^{(2^m)}$ are
\begin{gather*}\begin{array}{|c|c|c|c|}
 \hline
 & g^{(2^m)} & g^{[2^m]} & g^{[\frac{2^m}{\sqrt{2}}]} \\\hline
 k=0 & N+O_1 & 2N+O_1+2O_1 & 2N+O_1+2O_1 \\\hline
 k=1 & 2N+O_1 & 2N+O_1+2O_1 & 2N+O_1 \\\hline
 k\geq1 & 2^kN+O_1 & 2^kN+O_1+2^kO_2 & \left(2^{k+1}N+O_1+2^{k+1}O_2\right)^{\frac{1}{2}} \\\hline
 \end{array}
 \end{gather*}

 Applying $(n)$-replication to every entry in the table we obtain
\begin{gather*} \begin{array}{|c|c|c|c|}
 \hline
 & g^{(2^mn)} & \big(g^{[2^m]}\big)^{(n)} & \big(g^{[\frac{2^m}{\sqrt{2}}]}\big)^{(n)} \\\hline
 k=0 & N'+O'_1 & 2N'+O'_1+2O'_1 & 2N'+O'_1+2O'_1 \\\hline
 k=1 & 2N'+O'_1 & 2N'+O'_1+2O'_1 & 2N'+O'_1 \\\hline
 k\geq1 & 2^kN'+O'_1 & 2^kN'+O'_1+2^kO'_2 & \left(2^{k+1}N'+O'_1+2^{k+1}O'_2\right)^{\frac{1}{2}} \\\hline
 \end{array}
 \end{gather*}

 We can see that every row in the table corresponds to some identity from Lemma \ref{decomp} and identity (\ref{aaa}) is satisfied in this case.

ii) If $O_2\neq\varnothing$ and $O'_2\neq\varnothing$ then the possibilities, depending on $k$, for $g^{[2^m]}$ and $g^{[\frac{2^m}{\sqrt{2}}]}$ when we apply Lemma~\ref{decomp} to decompose $g^{(2^m)}$ are
\begin{gather*}
 \begin{array}{|c|c|c|c|}
 \hline
 & g^{(2^m)} & g^{[2^m]} & g^{[\frac{2^m}{\sqrt{2}}]} \\\hline
 k=1 & 2N+O_1+2O_1 & N+O_1 & (2N+O_1+2O_1 )^{\frac{1}{2}} \\\hline
 k\geq1 & 2^kN+O_1+2^kO_2 & 2^kN+O_1 & 2^{k+1}N+O_1+2^{k+1}O_2 \\\hline
 \end{array}\end{gather*}

 Applying $(n)$-replication to every entry in the table we obtain
\begin{gather*}\begin{array}{|c|c|c|c|}
 \hline
 & g^{(2^mn)} & \big(g^{[2^m]}\big)^{(n)} & \big(g^{[\frac{2^m}{\sqrt{2}}]}\big)^{(n)} \\\hline
 k=1 & 2N'+O'_1+2O'_1 & N'+O'_1 & (2N'+O'_1+2O'_1 )^{\frac{1}{2}} \\\hline
 k\geq1 & 2^kN'+O'_1+2^kO'_2 & 2^kN'+O'_1 & 2^{k+1}N'+O'_1+2^{k+1}O'_2 \\\hline
 \end{array}\end{gather*}

 We can see that every row in the table corresponds to some identity from Lemma \ref{decomp} and identity \eqref{aaa} is satisfied in this case.

iii) If $O_2\neq\varnothing$ and $O'_2=\varnothing$ then the possibilities, depending on $k$, for $g^{[2^m]}$ and $g^{[\frac{2^m}{\sqrt{2}}]}$ when we apply Lemma~\ref{decomp} to decompose $g^{(2^m)}$ are
\begin{gather*}
 \begin{array}{|c|c|c|c|}
 \hline
 & g^{(2^m)} & g^{[2^m]} & g^{[\frac{2^m}{\sqrt{2}}]} \\\hline
 k=1 & 2N+O_1+2O_1 & N+O_1 & (2N+O_1+2O_1 )^{\frac{1}{2}} \\\hline
 k\geq1 & 2^kN+O_1+2^kO_2 & 2^kN+O_1 & 2^{k+1}N+O_1+2^{k+1}O_2 \\\hline
 \end{array}\end{gather*}

 Applying $(n)$-replication to every entry in the table we obtain
\begin{gather*}
 \begin{array}{|c|c|c|c|c|}
 \hline
 & g^{(2^mn)} & \big(g^{[2^m]}\big)^{(n)} & \big(g^{[\frac{2^m}{\sqrt{2}}]}\big)^{(n)} \\\hline
 k=1 & 2N' & N' & (2N' )^{\frac{1}{2}} \\\hline
 k\geq1 & 2^kN'+O'_1 & 2^kN'+O'_1 & 2^{k+1}N'+O'_1 \\\hline
 \end{array}\end{gather*}

We can see that the first row corresponds to an identity from Lemma~\ref{decomp} and in the last row the Hauptmodul $2^{k+1}N'+O'_1$ satisfies the condition of Remark~\ref{conj1/2}. Identity~\eqref{aaa} is satisfied in both cases.

2.~If we are taking $g^{[2^m]}=g^{(2^m)}$ and $g^{[\frac{2^m}{\sqrt{2}}]}$ to be any Hauptmodul satisfying the condition in Remark~\ref{conj1/2} then identity~\eqref{aaa} is trivially satisfied because odd replication preserves the property of Remark \ref{conj1/2}.

 This completes the proof in the case $g$ is completely replicable.

The proof also shows that in Tables~\ref{Table1} and~\ref{Table2} we can take any $2^kN\,\parallel\, h+O_1+2^kO_2$ instead of $2^kN\,|\,h+O_1+2^kO_2$ and the result is still valid.

For the remaining cases in the tables, $g$ is not completely replicable, but its invariance group is conjugate by
$\left[\begin{smallmatrix}
1 & \alpha \\
0 & 1 \\
\end{smallmatrix}\right]$, with $\alpha=\frac{1}{2},\frac{1}{4}$ or $\frac{1}{8}$, to some group whose Hauptmodul $f$ is completely replicable. We shall give the proof for the case $\alpha=\frac{1}{2}$ as the other two
 cases are similar.

 Let $g=T_G$ and $f=T_{G^\alpha}$ so that $g(z)=-f\big(z+\frac{1}{2}\big)$. Since we are assuming $f$ is completely replicable it follows, as remarked above, that $f$ is \tp-replicable. So for some~-- not necessarily unique~-- \tp-replicates we have
 \begin{gather}
P_{n,f}(f(z))= \sum_{\substack{ad=n \\ 0\leq b <d}} f^{[a]}\left(\frac{az+b}{d}\right)+ \sum_{\substack{ad=n \\ d \ \text{even}\\ 0\leq b <d}} f^{[\sqrt{2}a]}\left(\frac{2az+b}{d}\right).
 \label{ftprep}
 \end{gather}

 Our aim is to use the relationship between $f$ and $g$ to show that (\ref{ftprep}) implies that $g$ is \tp-replicable after a suitable choice of \tp-replicates for~$f$.

 Now by the definition of the Faber polynomials, for any series $h(z) = \frac{1}{q} + \cdots$ we have $P_{n,h}\big(h\big(z+\frac{1}{2}\big)\big)=(-1)^nP_{n,h^\frac{1}{2}}\big(h^\frac{1}{2}(z)\big)$, where $h^\frac{1}{2}(z)=-h\big(z+\frac{1}{2}\big)$.

 In particular applying this to $f$ and $g$ gives $P_{n,f}\big(f\big(z+\frac{1}{2}\big)\big)=(-1)^nP_{n,g}(g)$. So for $n$ odd, substituting $z$ by $z+\frac{1}{2}$ in~(\ref{ftprep}) gives
 \begin{gather*}
 P_{n,g}(g(z)) =- \sum_{\substack{ad=n\\ 0\leq b <d}} f^{[a]}\left(\frac{az+b}{d}+\frac{a}{2d}\right)
 = \sum_{\substack{ad=n\\ 0\leq b <d}} (f^{[a]})^{\frac{1}{2}}\left(\frac{az+b+\frac{a-d}{2}}{d}\right)\\
\hphantom{P_{n,g}(g(z)) =}{} = \sum_{\substack{ad=n\\ 0\leq b <d}} (f^{[a]})^{\frac{1}{2}}\left(\frac{az+b}{d}\right).
 \end{gather*}

 So these replication identities for $g$ are satisfied by choosing $g^{[a]} = \big(f^{[a]}\big)^{\frac{1}{2}}$ for $a$ odd.

For $n$ even, we again substitute $z$ by $z+\frac{1}{2}$ in (\ref{ftprep}). Arranging the resulting right hand side by the highest power of 2 which divides $a$ gives
\begin{gather}
 P_{n,g}(g(z)) = \sum_{\substack{ad=n \\a \ \text{odd}\\ 0\leq b <d}} f^{[a]}\left(\frac{az+b}{d}+\frac{a}{2d}\right)+ \sum_{\substack{ad=n \\a \ \text{odd}\\ d \ \text{even}\\ 0\leq b <d}} f^{[\sqrt{2}a]}\left(\frac{2az+b}{d}\right) \nonumber\\
\hphantom{P_{n,g}(g(z)) =}{} + \sum_{i=1}^{\infty}\left(\sum_{\substack{ad=\frac{n}{2^i}\\ a \ \text{odd} \\ 0\leq b <d}}
 f^{[2^ia]}\left(\frac{2^iaz+b}{d}\right)
 + \sum_{\substack{2^iad=n \\ a \ \text{odd} \\d \ \text{even}\\ 0\leq b <d}}
 f^{[2^i\sqrt{2}a]}\left(\frac{2^{i+1}az+b}{d}\right)\right).\label{i=0}
\end{gather}

So if we set $g^{[2^\frac{n}{2}m]}=f^{[2^\frac{n}{2}m]}$ for $n\geq 2$ these replication identities for $g$ will be satisfied provided we can show that the $i=0$ term
\begin{gather}\label{zeroterm}
 \sum_{\substack{ad=n \\ a \ \text{odd}\\ 0\leq b <d}} f^{[a]}\left(\frac{az+b}{d}+\frac{a}{2d}\right)+ \sum_{\substack{ad=n \\a \ \text{odd}\\ d \ \text{even}\\ 0\leq b <d}} f^{[\sqrt{2}a]}\left(\frac{2az+b}{d}\right)
\end{gather}
has the form{\samepage
\begin{gather}\label{gform}
 \sum_{\substack{ad=n \\a \ \text{odd}\\ 0\leq b <d}} g^{[a]}\left(\frac{az+b}{d}\right)+ \sum_{\substack{ad=n \\a \ \text{odd}\\ d \ \text{even}\\ 0\leq b <d}} g^{[\sqrt{2}a]}\left(\frac{2az+b}{d}\right)
\end{gather}
with $g^{[a]}$ as given above and for suitable choices of the $g^{[\sqrt{2}a]}$.}

The proof is by considering cases. First consider $g=T_{(2N+O_1+2O_1)^\frac{1}{2}}$, thus $f=T_{2N+O_1+2O_1}$. Then each $f^{[a]}$ in (\ref{zeroterm}) is of type $T_{2M+O'_1+2O'_1}$, where $M=N/\gcd(N,a)$ and $O'_1=\{e\in O_1 \,|\, e \mbox{ divides } M\}$ and we know from Theorem~\ref{Feren} that in this case
\begin{gather*}
T_{2M+O'_1+2O'_1}\left(\frac{z}{2}\right)+T_{2M+O'_1+2O'_1}\left(\frac{z+1}{2}\right)+T_{2M+O'_1+2O'_1}(z)=T_{M+O'_1}(z).
\end{gather*}

In particular, substituting $z$ by $\frac{2az+b}{d}$ and summing over $b$ we have
\begin{gather*}
 \sum_{0\leq b<2d} T_{2M+O'_1+2O'_1}\left(\frac{az}{d}+\frac{b}{2d}\right)+\sum_{0\leq b<d}T_{2M+O'_1+2O'_1}\left(\frac{2az+b}{d}\right)\\
 \qquad{} =\sum_{0\leq b<d}T_{M+O'_1}\left(\frac{2az+b}{d}\right)
\end{gather*}
or, equivalently,
\begin{gather}
 \sum_{0\leq b<d} T_{2M+O'_1+2O'_1}\left(\frac{az+b}{d}+\frac{1}{2d}\right)+ \sum_{0\leq b<d} T_{2M+O'_1+2O'_1}\left(\frac{2az+b}{d}\right)\nonumber\\
\qquad{} =-\sum_{0\leq b<d}T_{2M+O'_1+2O'_1}\left(\frac{az+b}{d}\right)+\sum_{0\leq b<d}T_{M+O'_1}\left(\frac{2az+b}{d}\right)\nonumber\\
\qquad{} =\sum_{0\leq b<d}T_{(2M+O'_1+2O'_1)^{\frac{1}{2}}}\left(\frac{az+b}{d}+\frac{1}{2}\right)+\sum_{0\leq b<d}T_{M+O'_1}\left(\frac{2az+b}{d}\right)\nonumber\\
\qquad{} =\sum_{0\leq b<d}T_{(2M+O'_1+2O'_1)^{\frac{1}{2}}}\left(\frac{az+b}{d}\right)+\sum_{0\leq b<d}T_{M+O'_1}\left(\frac{2az+b}{d}\right). \label{fform}
\end{gather}

As mentioned above, the \tp-replicates of $f$ in (\ref{ftprep}) are not necessarily unique and we are free to make any convenient choices. So we now choose $f^{[\sqrt{2}a]}=T_{2M+O'_1+2O'_1}$. Identity~(\ref{fform}) then shows that (\ref{zeroterm}) has the form (\ref{gform}), with $g^{[a]}=\big(T_{2M+O'_1+2O'_1}\big)^{(\frac{1}{2})}$ and $g^{[\sqrt{2}a]}=T_{M+O'_1}$, as required.

\looseness=-1 For the second part of the theorem we just note that when $n\geq 2$ we have $g^{[2^{\frac{n}{2}}m]}=f^{[2^{\frac{n}{2}}m]}=\big(f^{[2^{\frac{n}{2}}]}\big)^{(m)}=\big(g^{[2^{\frac{n}{2}}]}\big)^{(m)}$, when $n=1$ we have $g^{[\sqrt{2}m]}=T_{M+O'_1}=\big(T_{N+O_1}\big)^{(m)}=\big(g^{[\sqrt{2}]}\big)^{(m)}$ and when $n=0$ we have \begin{gather*} g^{[m]}=\big(T_{2M+O'_1+2O'_1}\big)^{\frac{1}{2}}=\big(\big(T_{2N+O_1+2O_1}\big)^{(m)}\big)^{\frac{1}{2}}
=\big(\big(T_{2N+O_1+2O_1}\big)^{\frac{1}{2}}\big)^{(m)}=g^{(m)}.
\end{gather*}

This finishes the proof for functions of the form $T_{(2N+O_1+2O_1)^\frac{1}{2}}$.

We now consider $g$ of the form $T_{(2^kN+O_1+2^kO_2)^\frac{1}{2}}$, with $k\geq 2$. In this case, the $f^{[a]}$ in~(\ref{i=0}) are equal to $T_{2^kM+O'_1+2^kO'_2}$, where $M=N/\gcd(N,a)$, $O'_1=\{e\in O_1 \,|\, e \mbox{divides } M\}$ and $O'_2=\{e\in O_2 \,|\, e \mbox{ divides } M\}$ and applying Lemma~\ref{decomp} to this function we have
\begin{gather*}
T_{2^kM+O'_1+2^kO'_2}\left(\frac{z}{2}\right)+T_{2^kM+O'_1+2^kO'_2}\left(\frac{z+1}{2}\right)+T_{2^{k-1}M+O'_1}(z)=T_{2^{k-1}M+O'_1+2^{k-1}O'_2}(z).
\end{gather*}

Substituting $z$ by $\frac{2az+b}{d}$ and summing over $b$ we obtain
\begin{gather*}
\begin{split}&
 \sum_{0\leq b<2d}T_{2^kM+O'_1+2^kO'_2}\left(\frac{az}{d}+\frac{b}{2d}\right)+\sum_{0\leq b<d}T_{2^{k-1}M+O'_1}\left(\frac{2az+b}{d}\right)\\
 &\qquad{} =\sum_{0\leq b<d}T_{2^{k-1}M+O'_1+2O'_2}\left(\frac{2az+b}{d}\right) \end{split}
\end{gather*}
and, in particular,
\begin{gather*}
\sum_{0\leq b<d}{} T_{2^kM+O'_1+2^kO'_2}\left(\frac{az+b}{d} +\frac{1}{2d}\right)+\sum_{0\leq b<d}T_{2^{k-1}M+O'_1}\left(\frac{2az+b}{d}\right) \\
 \qquad{} =-\sum_{0\leq b<d}T_{2^kM+O'_1+2^kO'_2}\left(\frac{az+b}{d}\right)+\sum_{0\leq b<d}T_{2^{k-1}M+O'_1+2^{k-1}O'_2}\left(\frac{2az+b}{d}\right)\\
 \qquad{}= \sum_{0\leq b<d}T_{(2^kM+O'_1+2^kO'_2)^\frac{1}{2}}\left(\frac{az+b}{d}+\frac{1}{2}\right)+\sum_{0\leq b<d}T_{2^{k-1}M+O'_1+2^{k-1}O'_2}\left(\frac{2az+b}{d}\right)\\
 \qquad{} =\sum_{0\leq b<d}T_{(2^kM+O'_1+2^kO'_2)^\frac{1}{2}}\left(\frac{az+b}{d}\right)+\sum_{0\leq b<d}T_{2^{k-1}M+O'_1+2^{k-1}O'_2}\left(\frac{2az+b}{d}\right).
\end{gather*}

\looseness=-1 As mentioned above, the \tp-replicates of $f$ in (\ref{ftprep}) are not necessarily unique and we are free to make any convenient choices. So we now choose $f^{[\sqrt{2}a]}$ of the form $T_{2^{k-1}M+O'_1}$. Identity~(\ref{fform}) then shows that~(\ref{zeroterm}) has the form~(\ref{gform}), with $g^{[a]}=\big(T_{2M+O'_1+2O'_2}\big)^{\frac{1}{2}}$ and $g^{[\sqrt{2}a]}=T_{M+O'_1}$, as required.

For the second part of the theorem we just note that when $n\geq 2$ we have $g^{[2^{\frac{n}{2}}m]}=f^{[2^{\frac{n}{2}}m]}=\big(f^{[2^{\frac{n}{2}}]}\big)^{(m)}\!=\big(g^{[2^{\frac{n}{2}}]}\big)^{(m)}$, when $n=1$ we have $g^{[\sqrt{2}m]}=T_{2^{k-1}M+O'_1+2^{k-1}O'_2}=\big(T_{2^{k-1}N+O_1+2^kO_2}\big)^{(m)}\!$ $=\big(g^{[\sqrt{2}]}\big)^{(m)}$ and when $n=0$ we have $g^{[m]}=\big(T_{2^kM+O'_1+2^kO'_2}\big)^{\frac{1}{2}}=\big(\big(T_{2^kN+O_1+2^kO_2}\big)^{(m)}\big)^{\frac{1}{2}}=\big(\big(T_{2^kN+O_1+2^kO_2}\big)^{\frac{1}{2}}\big)^{(m)}=g^{(m)}$.

The remaining cases can be done similarly. For Hauptmoduls of groups of the form $G^{\frac{1}{4}}$ (resp.~$G^{\frac{1}{8}}$) it is the summand corresponding to $i=1$ (resp.\ $i=2$) that matters. We just note that the replicates $f^{[\sqrt{2}a]}$ (resp.\ $f^{[\sqrt{2}a]}$ and $f^{[2\sqrt{2}a]}$) with $a$ odd, have a power series expansion with coefficients of even powers of $q$ equal to zero. This means that any other function with the same property will work as well. We just chose those particular ones because they give complete \tp-replicability.
\end{proof}

\begin{Corollary}\label{subcolumn}
The functions $g$ on Tables~{\rm \ref{Table1}} and~{\rm \ref{Table2}} are completely \tp-replicable.
\end{Corollary}

\begin{proof}We begin the proof showing that we have complete \tp-replicability for powers of $\sqrt{2}$.

By complete \tp-replicability for powers of $\sqrt{2}$ we mean that, for any \tp-replicate of $g$,~$g^{[2^\frac{m}{2}n]}$, if we choose $\big(g^{[2^{\frac{m}{2}}n]}\big)^{[\sqrt{2}]}=g^{[2^{\frac{m+1}{2}}n]}$ and apply the \tp-replication rules from the theorem/tables we will necessarily have
\begin{gather}\label{commutation}
\big(g^{[2^{\frac{m}{2}}n]}\big)^{[2^{\frac{m'}{2}}]}=g^{[2^{\frac{m+m'}{2}}n]}, \qquad \text{for all} \ m'\geq 2.
\end{gather}

For doing this, we just have to see that after removing the first few entries of any column or applying $(m)$-replication, with $m$ odd, to a full column, the list we obtain is still a full column in the tables. This is a case by case check and is what guarantees that any replicate $g^{[2^\frac{k}{2}n]}$, together with $g^{[2^\frac{k+1}{2}n]}$ as its $[\sqrt{2}]$-replicate, will have \tp-replicates as expected
\begin{gather*}
\big(g^{[2^{\frac{m}{2}}n]}\big)^{[2^{\frac{m'}{2}}n']}
=\big(\big(g^{[2^{\frac{m}{2}}n]}\big)^{[2^{\frac{m'}{2}}]}\big)^{(n')} \quad \text{(by definition}) \\[-1mm]
\hphantom{\big(g^{[2^{\frac{m}{2}}n]}\big)^{[2^{\frac{m'}{2}}n']}}{}
=\big(g^{[2^{\frac{m+m'}{2}}n]}\big)^{(n')} \quad \text{(from (\ref{commutation})})\\[-1mm]
\hphantom{\big(g^{[2^{\frac{m}{2}}n]}\big)^{[2^{\frac{m'}{2}}n']}}{}
=\big(\big(g^{[2^{\frac{m+m'}{2}}]}\big)^{(n)}\big)^{(n')} \quad \text{(by definition)}\\[-1mm]
\hphantom{\big(g^{[2^{\frac{m}{2}}n]}\big)^{[2^{\frac{m'}{2}}n']}}{}
=\big(g^{[2^{\frac{m+m'}{2}}]}\big)^{(nn')} \quad \text{(with $g^{[2^{\frac{m+m'}{2}}]}$ comp.\ repl.\ or not)}\\[-1mm]
\hphantom{\big(g^{[2^{\frac{m}{2}}n]}\big)^{[2^{\frac{m'}{2}}n']}}{}
=g^{[2^{\frac{m+m'}{2}}nn']} \quad \text{(by definition)}.\tag*{\qed}
\end{gather*}\renewcommand{\qed}{}
\end{proof}

\section[Complete $(2{+})$-replicability and generalized Mahler recurrence relations]{Complete $\boldsymbol{(2{+})}$-replicability\\ and generalized Mahler recurrence relations}\label{mahler}

In this Section we extend some of Martin's results from \cite{Martin} to \tp-replication. Namely, we adapt his proof to show that the coefficients of completely \tp-replicable functions satisfy recurrence relations very similar to those of completely replicable functions.

We consider $\mathcal{L}$ a field extension of $\ratn$ containing all roots of unity, the ring $K\!=\!\mathcal{L}\big[\ldots,x_m^{[n]},\ldots\big]$, $m\in\nat$, $n\in\nat\cup\sqrt{2}\nat$, the series
\begin{gather*}
 h^{[r]}(q)=\frac{1}{q}+ \sum_{m=1}^{\infty}x_m^{[r]}q^m
\end{gather*}
for $r\in \nat\cup\sqrt{2}\nat$, and the polynomials $P_{k,r}(t)$ defined inductively by $P_{1,r}(t)=t$ and $P_{k,r}(t)=tP_{k-1,r}(t)- \sum\limits_{s=1}^{k-2}x_s^{[r]}P_{k-s-1,r}(t)-kx_{k-1}^{[k]}$. These are the same recurrence relations that Faber polynomials satisfy.

We fix $r\in\nat\cup\sqrt{2}\nat$ and consider the set of equations indexed by $k\geq 1$
\begin{gather}\label{comprep}
 \sum_{\substack{ad=k \\ 0\leq b<d}} h^{[ra]}\big(e^{2\pi i \frac{b}{d}}q^{\frac{a}{d}}\big)+
 \sum_{\substack{ad=k \\ d \ \text{even}\\0\leq b<d}} h^{[ra\sqrt{2}]}\big(e^{2\pi i \frac{b}{d}}q^{\frac{2a}{d}}\big)=P_{k,r}\big(h^{[r]}(q)\big).
\end{gather}

These equations give an infinite set of identities in $K$ by equating the coefficients of equal powers of $q$ in both sides of the each equation.
We denote by $I^{[r]}$ the ideal in $K$ generated by them and write $I$ for the ideal in $K$ generated by $ \bigcup\limits_{r\in\nat\cup\sqrt{2}\nat}I^{[r]}$.

If $f(q)=\frac{1}{q}+\sum\limits_{k=1}^{\infty} a_kq^k$ is completely \tp-replicable with replicates $f^{[n]}(q)=\frac{1}{q}+ \sum\limits_{k=1}^{\infty} a_k^{[n]}q^k$, $n\in\nat\cup\sqrt{2}\nat$ then they satisfy equations (\ref{comprep}) with $h^{[n]}(q)$ and $h^{[\sqrt{2}n]}(q)$ replaced by $f^{[n]}(q)$ and $f^{[\sqrt{2}n]}(q)$, respectively. This means that every completely \tp-replicable function $f$ induces a~non-trivial homomorphism $E_f\colon K\longrightarrow \mathbb{C}$ with $E_f\big(x_k^{[n]}\big)=a_k^{[n]}$, whose kernel contains~$I$.

For $u\in\nat\cup\sqrt{2}\nat$ we define a $\mathcal{L}$-algebra endomorphism $\psi_u$ of $K$ letting it fix every element of $\mathcal{L}$ and mapping $x_m^{[n]}$ to $x_m^{[nu]}$. Since the equations defining $I^{[r]}$ and $I^{[ru]}$ have the same form, it is clear that $\psi_u(I^{[r]})=I^{[ru]}$ for any $r$. Consequently, $\psi_u(I)\subseteq I$ and we think of $\psi_u$ as an $\mathcal{L}$-algebra endomorphism of the quotient $K/I$.

For each $M\geq 1$ we set $R_M=K/I\big[\big[q^{\frac{1}{2M}}\big]\big]$ and also \begin{gather*}\Delta=\left\{\left[
\begin{matrix}
a & b \\
0 & d
\end{matrix}
\right]|\, a,b,d\in\nat\right\}\cup
\left\{\left[
\begin{matrix}
\sqrt{2}a & b/\sqrt{2} \\
0 & \sqrt{2}d
\end{matrix}
\right]|\, a,b,d\in\nat\right\}.\end{gather*}

We define an action of $\Delta$ in $R_M$ in the following way. For $\alpha=\left[\begin{smallmatrix}
u & v \\
0 & y \\
\end{smallmatrix}\right]\in \Delta$ we set $e\parallel \alpha= e$, $x_m^{[r]}\parallel\alpha=\psi_u\big(x_m^{[r]}\big)=x_m^{[ru]}$ and $q^{\frac{1}{M}}\parallel\alpha=e^{2\pi i \frac{v}{My}}q^{\frac{u}{My}}$. Then we extend this map to an $\mathcal{L}$-homomorphism from $K\big[\big[q^{\frac{1}{M}}\big]\big]$ to $K\big[\big[q^{\frac{1}{My}}\big]\big]$. Though the context should make it clear, we warn the reader that in this Chapter only the symbol $\parallel$ has this meaning.

\begin{Remark}The ideal $I$ of $K$ is stable under $\parallel\alpha$ for every $\alpha\in \Delta$.
\end{Remark}

\begin{Remark} For every $\alpha,\beta\in\Delta$ and $h(q)\in R_M$ we have \begin{gather*}(h(q)\parallel\alpha)\parallel \beta=h(q)\parallel \alpha\beta.\end{gather*}
\end{Remark}

\begin{Remark}
 If $\alpha=\left[
\begin{smallmatrix}
u & x \\
0 & y \\
\end{smallmatrix}
\right]\in\Delta$ then $h^{[r]}(q)\parallel\alpha=h^{[ru]}\big(e^{2\pi i \frac{x}{y}}q^{\frac{u}{y}}\big)$.
\end{Remark}

\begin{Remark}For every positive integer $j$, $\big(h^{[r]}(q)\big)^j\parallel\alpha=\big(h^{[r]}(q)\parallel\alpha\big)^j$ as $\parallel\alpha$ is a ring homo\-morphism.
\end{Remark}

\begin{Definition}
Let $n$ be a positive integer and $h(q)\in R_1$. We define ${\bf T}_n$ by
\begin{gather*}
 {\bf T}_n(h(q))= \sum_{\substack{uy=n\\0\leq v<y}} h(q)\parallel\left[\begin{matrix}
u & v \\
0 & y \\
\end{matrix}\right]+
\sum_{\substack{uy=\frac{n}{2}\\0\leq v<2y}} h(q)\parallel\left[\begin{matrix}
\sqrt{2}u & \dfrac{v}{\sqrt{2}} \\
 0 & \sqrt{2}y \\
\end{matrix}\right]
\end{gather*}
and call ${\bf T}_n$ a generalized Hecke operator. Both sums are over positive integer numbers which makes the second sum be zero if $n$ is odd.
\end{Definition}

With this definition equation (\ref{comprep}) becomes
\begin{gather}\label{comprep2}
{\bf T}_k\big(h^{[r]}(q)\big)=P_{k,r}\big(h^{[r]}(q)\big).
\end{gather}

\begin{Definition} Let $n\in\nat\cup\sqrt{2}\nat$ and $h(q)\in R_1$. We denote by $\Psi_n$ the mapping $\Psi_n h(q)=h(q)\parallel
\left[\begin{smallmatrix}
n & 0 \\
0 & n
\end{smallmatrix}\right]$.
\end{Definition}

\begin{Proposition} Let $l_1$ and $l_2$ be relatively prime positive integers and $h(q)\in R_1$. Then
\begin{gather*}{\bf T}_{l_1}{\bf T}_{l_2}h(q)={\bf T}_{l_1l_2}h(q).\end{gather*}
In particular, ${\bf T}_{l_1}$ and ${\bf T}_{l_2}$ commute.
\end{Proposition}

\begin{proof}This comes from the fact that
\begin{gather*}
\left\{
\left[\begin{matrix}
u & v \\
0 & y \\
\end{matrix}\right]
\left[\begin{matrix}
u'& v'\\
0 & y'\\
\end{matrix}\right]
 |\, uy =l_1,\, u'y'=l_2,\, 0\leq v<y, \, 0\leq v'<y',\, (yy' \ \text{even})\right\}
\\
 \qquad = \left\{
\left[\begin{matrix}
u'' & v'' \\
0 & y'' \\
\end{matrix}\right]
|\,u''v''=l_1l_2,\, 0\leq v''<y'',\, (y'' \ \text{even})\right\}.\tag*{\qed}
\end{gather*}\renewcommand{\qed}{}
\end{proof}

\begin{Proposition}If $p$ is an odd prime then \begin{gather*}{\bf T}_{p^n}{\bf T}_{p}(h(q))={\bf T}_{p^{n+1}}(h(q))+p{\bf T}_{p^{n-1}}\Psi_{p}(h(q)).\end{gather*} If $p=2$ then
\begin{gather*}{\bf T}_{2^n}{\bf T}_{2}(h(q))={\bf T}_{2^{n+1}}(h(q))+2{\bf T}_{2^{n}}\Psi_{\sqrt{2}}(h(q))+2{\bf T}_{2^{n-1}}\Psi_{2}(h(q)).\end{gather*}
\end{Proposition}

\begin{proof}The case where $p$ is odd is true as the operators ${\bf T}_{p^n}$ agree with the Hecke operator for ${\rm PSL}_2(\intn)$.
For the case $p=2$ we see that ${\bf T}_{2^n}{\bf T}_2(h(q))$ equals
\begin{gather*}
h(q)\parallel
\left(
\sum_{\substack{i=0,1\\ 0\leq b<2^i}}\left[\begin{matrix}
2^{1-i} & b \\
0 & 2^i
\end{matrix}
\right]
+
\sum_{0\leq b<2}\left[\begin{matrix}
\sqrt{2} & \frac{b}{\sqrt{2}} \\
0 & \sqrt{2}
\end{matrix}
\right]
\right)
\left(
\sum_{\substack{i=0,\ldots,n\\ 0\leq b<2^i}}\left[\begin{matrix}
2^{n-i} & b \\
0 & 2^i
\end{matrix}
\right]
\right)
\\
\qquad{} +h(q)\parallel
\left(
\sum_{\substack{i=0,1\\ 0\leq b<2^i}}\left[\begin{matrix}
2^{1-i} & b \\
0 & 2^i
\end{matrix}
\right]+
\sum_{0\leq b<2}\left[\begin{matrix}
\sqrt{2} & \frac{b}{\sqrt{2}} \\
0 & \sqrt{2}
\end{matrix}
\right]
\right)
\left(
\sum_{\substack{i=0,\ldots,n-1\\ 0\leq b<2^{i+1}}}\left[\begin{matrix}
2^{n-1-i}\sqrt{2} & \frac{b}{\sqrt{2}} \\
0 & 2^i\sqrt{2} \\
\end{matrix}
\right]
\right)\\
{}=h(q)\parallel
 \!\left(
\sum_{\substack{i=0,\ldots,n\\0\leq b<2^i}}
\begin{bmatrix}
2^{n+1-i} & 2b \\
0 & 2^i
\end{bmatrix}
+
\sum_{\substack{i=0,\ldots,n\\0\leq b<2^i}}
\begin{bmatrix}
2^{n-i}\sqrt{2} & \sqrt{2}b \\
0 & 2^i\sqrt{2}
\end{bmatrix}
\right.
+
\sum_{{\substack{i=0,\ldots,n\\0\leq b<2^i}}}
\begin{bmatrix}
2^{n-i}\sqrt{2} & \sqrt{2}b+\frac{2^i}{\sqrt{2}}\\
0 & 2^i\sqrt{2}
\end{bmatrix}
\\
\qquad{} +
\sum_{\substack{i=0,\ldots,n\\0\leq b<2^{i}}}
\begin{bmatrix}
2^{n-i} & b \\
0 & 2^{i+1}
\end{bmatrix}
+
\sum_{\substack{i=0,\ldots,n\\0\leq b<2^{i}}}
\begin{bmatrix}
2^{n-i} & b+2^i \\
0 & 2^{i+1}
\end{bmatrix}
 +
\sum_{{{\substack{{i=0,\ldots,n-1} \\ {0\leq b<2^{i+1}} }}}}
\begin{bmatrix}
2^{n-i}\sqrt{2} & \sqrt{2}b \\
0 & 2^i\sqrt{2}
\end{bmatrix}
\\
\qquad{} +
\sum_{\substack{i=0,\ldots,n-1\\0\leq b<2^{i+1}}}
\begin{bmatrix}
2^{n-i} & b \\
0 & 2^{i+1}
\end{bmatrix}
+
\sum_{\substack{i=0,\ldots,n-1\\0\leq b<2^{i+1}}}
\begin{bmatrix}
2^{n-i} & b+2^i \\
0 & 2^{i+1}
\end{bmatrix}\\
\qquad{} +\sum_{{\substack{i=0,\ldots,n-1\\0\leq b<2^{i+1}}}}
\begin{bmatrix}
2^{n-1-i}\sqrt{2} & \frac{b}{\sqrt{2}} \\
0 & 2^{i+1}\sqrt{2}
\end{bmatrix}+
\left.
\sum_{\substack{i=0,\ldots,n-1\\0\leq b<2^{i+1}}}
\begin{bmatrix}
2^{n-1-i}\sqrt{2} & \frac{b}{\sqrt{2}}+2^ib \\
0 & 2^{i+1}\sqrt{2} \\
\end{bmatrix}
\right)
\\
 {} =h(q)\parallel
\left(
\sum_{\substack{i=0,\ldots,n\\0\leq b<2^i}}\left[\begin{matrix}
2^{n+1-i} & 2b \\
0 & 2^i
\end{matrix}
\right]+
\sum_{\substack{i=0,\ldots,n\\0\leq b<2^i}}\left[\begin{matrix}
2^{n-i}\sqrt{2} & \sqrt{2}b \\
0 & 2^i\sqrt{2}
\end{matrix}
\right]\right.\\
\qquad{} +
\sum_{\substack{i=0,\ldots,n\\0\leq b<2^i}}\left[\begin{matrix}
2^{n-i}\sqrt{2} & \sqrt{2}b+\frac{2^i}{\sqrt{2}}\\
0 & 2^i\sqrt{2}
\end{matrix}
\right]+
\sum_{\substack{i=0,\ldots,n\\0\leq b<2^{i+1}}}\left[\begin{matrix}
2^{n-i} & b \\
0 & 2^{i+1}
\end{matrix}
\right]\\
\qquad{}+
2 \sum_{\substack{i=0,\ldots,n-1\\0\leq b<2^{i}}}\left[\begin{matrix}
2^{n-i}\sqrt{2} & \sqrt{2}b \\
0 & 2^i\sqrt{2}
\end{matrix}
\right]+
2 \sum_{\substack{i=0,\ldots,n-1\\0\leq b<2^{i+1}}}\left[\begin{matrix}
2^{n-i} & b \\
0 & 2^{i+1}
\end{matrix}
\right]\\
\left.\qquad{} +
\sum_{\substack{i=0,\ldots,n-1\\0\leq b<2^{i+2}}}\left[\begin{matrix}
2^{n-1-i}\sqrt{2} & \frac{b}{\sqrt{2}} \\
0 & 2^{i+1}\sqrt{2}
\end{matrix}
\right]
\right).
\end{gather*}
Now, the first and fourth summands equal
\begin{gather*}
h(q)\parallel
\left(
 \sum_{\substack{i=0,\ldots,n+1\\0\leq b<2^i}}\left[\begin{matrix}
2^{n+1-i} & b \\
0 & 2^i
\end{matrix}
\right]+
2 \sum_{\substack{i=0,\ldots,n-1\\0\leq b<2^i}}\left[\begin{matrix}
2^{n-i} & 2b \\
0 & 2^i \\
\end{matrix}
\right]
\right),
\end{gather*}
the second, third and last summands equal
\begin{gather*}
h(q)\parallel
\left(
 \sum_{\substack{i=0,\ldots,n\\0\leq b<2^{i+1}}}\left[\begin{matrix}
2^{n-i}\sqrt{2} & \frac{b}{\sqrt{2}} \\
0 & 2^i\sqrt{2}
\end{matrix}
\right]+
2 \sum_{\substack{i=0,\ldots,n-2\\0\leq b<2^{i+1}}}\left[\begin{matrix}
2^{n-2-i} & \sqrt{2}b \\
0 & 2^{i+1}\sqrt{2}
\end{matrix}
\right]
\right),
\end{gather*}
and this shows that
\begin{gather*}
{\bf T}_{2^n}{\bf T}_{2}(h(q)) = h(q)\parallel
\left(
\sum_{\substack{i=0,\ldots,n+1\\0\leq b<2^{i}}}
\begin{bmatrix}
2^{n+1-i} & b \\
0 & 2^i
\end{bmatrix}
\right. +
\sum_{\substack{i=0,\ldots,n\\0\leq b<2^{i+1}}}
\begin{bmatrix}
2^{n-i}\sqrt{2} & \frac{b}{\sqrt{2}}\\
0 & 2^{i}\sqrt{2}
\end{bmatrix}
\\
\hphantom{{\bf T}_{2^n}{\bf T}_{2}(h(q))=}{}+
2 \sum_{\substack{i=0,\ldots,n-1\\0\leq b<2^{i+1}}}
\begin{bmatrix}
2^{n-i}\sqrt{2} & \sqrt{2}b \\
0 & 2^i\sqrt{2}
\end{bmatrix} + 2 \sum_{\substack{i=0,\ldots,n-1\\0\leq b<2^{i+1}}}
\begin{bmatrix}
2^{n-i} & b \\
0 & 2^{i+1}
\end{bmatrix}\\
 \left. \hphantom{{\bf T}_{2^n}{\bf T}_{2}(h(q))=}{}
 + 2 \sum_{\substack{i=0,\ldots,n-1\\0\leq b<2^{i}}}
\begin{bmatrix}
2^{n-i} & 2b \\
0 & 2^{i+1}
\end{bmatrix}
+2 \sum_{\substack{i=0,\ldots,n-2\\0\leq b<2^{i+1}}}
\begin{bmatrix}
2^{n-2-i} & \sqrt{2}b \\
0 & 2^{i+1}\sqrt{2}
\end{bmatrix}
\right)\\
\hphantom{{\bf T}_{2^n}{\bf T}_{2}(h(q))}{}
= {\bf T}_{2^{n+1}}(h(q))+2{\bf T}_{2^{n}}\Psi_{\sqrt{2}}(h(q))+2{\bf T}_{2^{n-1}}\Psi_{2}(h(q)),
\end{gather*}
and the theorem is proven.
\end{proof}

\begin{Corollary} The algebra generated by the operators ${\bf T}_{n}$, for $n\in\nat$, is commutative.
\end{Corollary}

\looseness=-1 Let $l$ be a fixed prime. We set $Q_k={\bf T}_l\big(h(q)^k\big)$ for $k\geq 1$, and for $b\in\frac{K}{I}$ we use $b^{[l]}$ to denote $b\parallel
\left[\begin{smallmatrix}
l & \ast \\
0 & \ast
\end{smallmatrix}\right]$.
We set also $Q_0=
 \begin{cases}
2l+1,& \text{if } l=2,\\
l+1,& \text{if } l\neq2
\end{cases}$ and define ${\bf T}_{\frac{k}{l}}$ as the operator that sends~$R_1$ to zero if~$l$ does not divide~$k$. We use the same notation to denote both $P_{k,r}(t)$ and its image in~$\frac{K}{I}[t]$.

\begin{Proposition}\label{rec}
For $k\in\nat$ write $P_{k,r}(t)=t^k+ \sum\limits_{i=1}^{k}b_{k,i}t^{k-i}\in\frac{K}{I}[t]$. If $l$ is odd then
\begin{gather*}Q_k+ \sum_{i=1}^kb_{k,i}Q_{k-i}+ \sum_{i=1}^k\big(b_{k,i}^{[l]}-b_{k,i}\big)h^{[l]}(q^l)^{k-i}=P_{kl}(h(q))+k{\bf T}_{\frac{k}{l}}\Psi_l(h(q)),\end{gather*}
 and if $l=2$ we have
\begin{gather*}
Q_k+ \!\sum_{i=1}^k \! b_{k,i}Q_{k-i}\! + \!\sum_{i=1}^k\!\big(b_{k,i}^{[\sqrt{2}]}\!-b_{k,i}\big)\big(h^{[\sqrt{2}]}(q)^{k-i}\!+h^{[\sqrt{2}]}(-q)^{k-i}\big)\!
 +\!\sum_{i=1}^k\!\big(b_{k,i}^{[2]}\!-b_{k,i}\big)h^{[2]}\big(q^2\big)^{k-i} \\
 \qquad{} = \begin{cases}
 P_{2k,1}(h(q)), & \mbox{if } 2\!\nmid\! k, \\
 P_{2k,1}(h(q))+2{\bf T}_{k}\Psi_{\sqrt{2}}(h(q))+2{\bf T}_{\frac{k}{2}}\Psi_{2}(h(q)), & \mbox{if } 2\,|\, k \\
 \end{cases}
\end{gather*}
\end{Proposition}
\begin{proof}
The case where $l$ is odd can be found in \cite{Martin}. When $l=2$ we apply ${\bf T}_2$ to both sides of equation (\ref{comprep2}) \begin{gather*}{\bf T}_2{\bf T}_k(h(q))={\bf T}_2\left(h(q)^k+ \sum_{i=1}^{k}b_{k,i}h(q)^{k-i}\right),\end{gather*}
and the following manipulation
\begin{gather*}
{\bf T}_2\big(b_{k,i}h(q)^{k-i}\big) =
b_{k,i}h(q)^{k-i}\parallel\left[\begin{matrix}
 2 & 0 \\
 0 & 1
 \end{matrix}\right]+
b_{k,i}h(q)^{k-i}\parallel\left[\begin{matrix}
 \sqrt{2} & 0 \\
 0 & \sqrt{2}
 \end{matrix}\right] \\
\hphantom{{\bf T}_2\big(b_{k,i}h(q)^{k-i}\big) =}{} +
b_{k,i}h(q)^{k-i}\parallel\left[\begin{matrix}
 \sqrt{2} & \frac{1}{\sqrt{2}} \\
 0 & \sqrt{2}
 \end{matrix}\right] +
b_{k,i}h(q)^{k-i}\parallel\left[\begin{matrix}
 1 & 0 \\
 0 & 2
 \end{matrix}\right] \\
\hphantom{{\bf T}_2\big(b_{k,i}h(q)^{k-i}\big) =}{}+
b_{k,i}h(q)^{k-i}\parallel\left[\begin{matrix}
 1 & 1 \\
 0 & 2
 \end{matrix}\right]\\
\hphantom{{\bf T}_2\big(b_{k,i}h(q)^{k-i}\big)}{} =
(b^{[2]}_{k,i}-b_{k,i})\left(h(q)^{k-i}\parallel\left[\begin{matrix}
 2 & 0 \\
 0 & 1
 \end{matrix}\right]\right) \\
\hphantom{{\bf T}_2\big(b_{k,i}h(q)^{k-i}\big) =}{}+
(b^{[\sqrt{2}]}_{k,i}-b_{k,i})\left(h(q)^{k-i}\parallel\left[\begin{matrix}
 \sqrt{2} & 0 \\
 0 & \sqrt{2}
 \end{matrix}\right]\right) \\
\hphantom{{\bf T}_2\big(b_{k,i}h(q)^{k-i}\big) =}{} +
(b^{[\sqrt{2}]}_{k,i}-b_{k,i})\left(h(q)^{k-i}\parallel\left[\begin{matrix}
 \sqrt{2} & \frac{1}{\sqrt{2}} \\
 0 & \sqrt{2}
 \end{matrix}\right]\right)+ b_{k,i}{\bf T}_2\big(h(q)^{k-i}\big)\\
\hphantom{{\bf T}_2\big(b_{k,i}h(q)^{k-i}\big) }{}
=\big(b^{[2]}_{k,i}-b_{k,i}\big)h^{[2]}\big(q^2\big)^{k-i}+\big(b^{[\sqrt{2}]}_{k,i}-b_{k,i}\big) \\
\hphantom{{\bf T}_2\big(b_{k,i}h(q)^{k-i}\big) =}{}
\times \big(h^{[\sqrt{2}]}(q)^{k-i}+h^{[\sqrt{2}]}(-q)^{k-i}\big)+b_{k,i}Q_{k-i}
\end{gather*}
shows that ${\bf T}_2{\bf T}_k(h(q))$ is equal to
\begin{gather*}
Q_k+ \sum_{i=1}^kb_{k,i}Q_{k-i}+ \sum_{i=1}^k\left(\big(b^{[2]}_{k,i}-b_{k,i}\big)h^{[2]}\big(q^2\big)^{k-i}+\big(b^{[\sqrt{2}]}_{k,i}-b_{k,i}\big)\right. \\
\left.\qquad{}\times\big(h^{[\sqrt{2}]}(q)^{k-i}+h^{[\sqrt{2}]}(-q)^{k-i}\big)\right).
\end{gather*}
If $(2,k)=1$ then ${\bf T}_2{\bf T}_k(h(q))={\bf T}_{2k}(h(q))=P_{2k,1}(h(q))$ and if $2^r$ is the exact power of $2$ that divides $k$ we have that
\begin{gather*}
{\bf T}_2{\bf T}_k(h(q)) ={\bf T}_{\frac{k}{2^r}}{\bf T}_{2^r}{\bf T}_2
 ={\bf T}_{\frac{k}{2^r}}\big({\bf T}_{2^{r+1}}(h(q))+2{\bf T}_{2^{r}}\Psi_{\sqrt{2}}(h(q))+2{\bf T}_{2^{r-1}}\Psi_{2}(h(q))\big)\\
\hphantom{{\bf T}_2{\bf T}_k(h(q))}{}
={\bf T}_{2k}(h(q))+2{\bf T}_{k}\Psi_{\sqrt{2}}(h(q))+2{\bf T}_{\frac{k}{2}}\Psi_{2}(h(q)).
\end{gather*}
This proves the assertion of the theorem.
\end{proof}

When $l=2$ the $Q_k$ are simply the power sum symmetric functions on $h^{[2]}\big(q^2\big)$, $h^{[\sqrt{2}]}(q)$, $h^{[\sqrt{2}]}(-q)$, $h(q)$, $h(-q)$. By induction, using the previous result, one shows that every $Q_j$ is a~polynomial in $h^{[2]}(q)$, $h^{[2]}\big(q^2\big)$, $h^{[\sqrt{2}]}(q)$, $h^{[\sqrt{2}]}(-q)$ and $h(q)$. Let
\begin{gather*}
\sigma_n= \sum_{1\leq i_1<i_2<\cdots<i_n\leq 5}x_{i_1}\cdots x_{i_n}
\end{gather*}
be the elementary symmetric functions in the indeterminates $x_1$, $x_2$, $x_3$, $x_4$, $x_5$. We know that the elementary symmetric function are polynomials in the power sum symmetric functions and from this we conclude that the elementary symmetric functions on $h^{[2]}\big(q^2\big)$, $h^{[\sqrt{2}]}(q)$, $h^{[\sqrt{2}]}(-q)$, $h(q)$, $h(-q)$ are polynomials in $h^{[2]}(q)$, $h^{[2]}\big(q^2\big)$, $h^{[\sqrt{2}]}(q)$, $h^{[\sqrt{2}]}(-q)$ and $h(q)$. We use these facts in the proof of the following proposition.

\begin{Proposition}If $f(q)=\frac{1}{q}+ \sum\limits_{k=1}^\infty a_kq^k$ is completely \tp-replicable with replicates $f^{[n]}(q)=\frac{1}{q}+ \sum\limits_{k=1}^\infty a_k^{[n]}q^k$, for $n\in\nat\cup\sqrt{2}\nat$, then
\begin{gather*}\begin{split}&
\sigma_2\left(f^{[2]}(2z),f^{[\sqrt{2}]}(z),f^{[\sqrt{2}]}\left(z+\frac{1}{2}\right),f\left(\frac{z}{2}\right),f\left(\frac{z+1}{2}\right)\right)\\
& \qquad{} =2a_2f(z)-f^{[2]}(z)+2\left(a_4-a_1\right)+2a_2^{[\sqrt{2}]}-\big(f^{[\sqrt{2}]}(z)\big)^2.
\end{split}
\end{gather*}
\end{Proposition}
\begin{proof}
We take $h^{[n]}(q)$, for $n\in \nat\cup\sqrt{2}\nat$ as before and start by noticing that
\begin{gather*} \sigma_2\left(h^{[2]}(2z),h^{[\sqrt{2}]}(z),h^{[\sqrt{2}]}\left(z+\frac{1}{2}\right),h\left(\frac{z}{2}\right),
h\left(\frac{z+1}{2}\right)\right)=\frac{1}{2}\big(Q_1^2-Q_2\big).
\end{gather*}

Then we use Proposition \ref{rec} to express $Q_1$ and $Q_2$ as polynomials in $h^{[2]}(q)$, $h^{[2]}\big(q^2\big)$, $h^{[\sqrt{2}]}(q)$, $h^{[\sqrt{2}]}(-q)$ and $h(q)$.

Now, \begin{gather*}Q_1={\bf T}_1(h(q))=P_1(h(q))=h^2(q)-2x_1,\end{gather*} and from Proposition~\ref{rec}
\begin{gather*}
Q_2=P_4(h(q))+2P_2\big(h^{[\sqrt{2}]}(q)\big)\!+2h^{[2]}(q)-\big({-}10x_1+2\big({-}2x_1^{[\sqrt{2}]}\!+2x_1\big)+\big({-}2x_1^{[2]}\!+2x_1\big)\big),
\end{gather*}
because $b_{2,1}=b_{2,1}^{[\sqrt{2}]}=b_{2,1}^{[2]}=0$ and $b_{2,1}=-2x_1$, $b_{2,1}^{[\sqrt{2}]}=-2x_1^{[\sqrt{2}]}$, $b_{2,1}^{[2]}=-2x_1^{[2]}$, and $\frac{1}{2}\big(Q_1^2-Q_2\big)$ becomes
\begin{gather*}
\frac{1}{2}(h^4(q)-4x_1h^2(q)+4x_1^2-\big(h^4(q)-4x_1h^2(q)-4x_2h(q)-4x_3+2x_1^2 \\
\qquad\quad{}+ 2h^{[\sqrt{2}]}(q)^2-4x_1^{[\sqrt{2}]}+2h^{[2]}(q)+10x_1+4x_1^{[\sqrt{2}]}-4x_1+2x_1^{[2]}-2x_1\big)\\
\qquad{} =2x_2h(q)-\big(h^{[\sqrt{2}]}(q)\big)^2-h^{[2]}(q)+x_1^2+2x_3-2x_1-x_1^{[2]}.
\end{gather*}

Applying the homomorphism $E_f$ defined at the beginning of this section we get
\begin{gather*}
\sigma_2\left(f^{[2]}(2z),f^{[\sqrt{2}]}(z),f^{[\sqrt{2}]}\left(z+\frac{1}{2}\right),f\left(\frac{z}{2}\right),f\left(\frac{z+1}{2}\right)\right)\\
\qquad{} =2a_2f(q)-\big(f^{[\sqrt{2}]}(q)\big)^2-f^{[2]}(q)+a_1^2+2a_3-2a_1-a_1^{[2]}.
\end{gather*}

Equating the coefficient of $q^2$ in both sides of the equation
\begin{gather*}
f^{[2]}\big(q^2\big)+f^{[\sqrt{2}]}(q)+f^{[\sqrt{2}]}(-q)+f\big(q^\frac{1}{2}\big)+f\big({-}q^\frac{1}{2}\big)=P_2(f(q)),
\end{gather*}
we see that $a_1^2+2a_3-a_1^{[2]}=2a_4+2a_2^{[\sqrt{2}]}$ and this concludes the proof.
\end{proof}

\begin{Theorem}\label{MahlerRec}
If $f(q)=\frac{1}{q}+ \sum\limits_{k=1}^\infty a_kq^k$ is completely \tp-replicable with replicates $f^{[n]}(q)=\frac{1}{q}+ \sum\limits_{k=1}^\infty a_k^{[n]}q^k$, for $n\in\nat\cup\sqrt{2}\nat$, then their coefficients satisfy the following recurrence relation:
\begin{alignat*}{3}
& 1) \quad && a_{4k}=a_{2k+1}+ \sum_{j=1}^{k-1}a_ja_{2k-j}+\frac{1}{2}\big(a_k^2-a_k^{[2]}\big)-a_{2k}^{[\sqrt{2}]},&\\
& 2) \quad && a_{4k+1}= a_{2k+3} -a_2a_{2k} + \sum_{j=1}^{k}a_ja_{2k+2-j} +\sum_{j=1}^{k-1}a_j^{[2]}a_{2k-2j}^{[\sqrt{2}]} +2\sum_{j=1}^{k-1}a_{4j}a_{2k-2j}^{[\sqrt{2}]}&\\
&&& \hphantom{a_{4k+1}=}{} + \sum_{j=1}^{k-1}a_{4j}a_{k-j}^{[2]} +\sum_{j=1}^{2k-1}(-1)^ja_ja_{4k-j} +\sum_{j=1}^{k-1}a_{2j}^{[\sqrt{2}]}a_{2k-2j}^{[\sqrt{2}]}& \\
&&& \hphantom{a_{4k+1}=}{} +\frac{1}{2}\big(a_{k+1}^2-a_{k+1}^{[2]}+a_{2k}^2+a_{2k}^{[2]}\big), &\\
& 3) \quad && a_{4k+2}= \sum_{j=1}^{k}a_ja_{2k+1-j}+a_{2k+2},&\\
& 4) \quad && a_{4k+3}=a_{2k+4}-a_2a_{2k+1}-\frac{1}{2}\big(a_{2k+1}^2-a_{2k+1}^{[2]}\big)+ \sum_{j=1}^{2k}(-1)^ja_ja_{4k+2-j}+a_{2k+2}^{[\sqrt{2}]}&\\
&&& \hphantom{a_{4k+3}=}{} + \sum_{j=1}^{k}a_{4j-2}a_{k+1-j}^{[2]}+\sum_{j=1}^{k+1}a_ja_{2k+3-j}+2\sum_{j=1}^{2k}a_{2j}a_{2k+1-j}^{[\sqrt{2}]}
+\sum_{j=1}^{k}a_j^{[\sqrt{2}]}a_{2k+1-j}^{[\sqrt{2}]}.&
\end{alignat*}
\end{Theorem}

\begin{proof}
This is a consequence of the following two identities
\begin{gather*}
\sigma_1\left(f^{[2]}(2z),f^{[\sqrt{2}]}(z),f^{[\sqrt{2}]}\left(z+\frac{1}{2}\right),f\left(\frac{z}{2}\right),
f\left(\frac{z+1}{2}\right)\right)=P_{2,f}(f(z)),\\
\sigma_2\left(f^{[2]}(2z),f^{[\sqrt{2}]}(z),f^{[\sqrt{2}]}\left(z+\frac{1}{2}\right),f\left(\frac{z}{2}\right),
f\left(\frac{z+1}{2}\right)\right)\\
\qquad {} =2a_2f(z)-f^{[2]}(z)+2\left(a_4-a_1\right)+2a_2^{[\sqrt{2}]}-\big(f^{[\sqrt{2}]}(z)\big)^2.\tag*{\qed}
\end{gather*}\renewcommand{\qed}{}
\end{proof}

\section[The baby monster Lie algebra and $(2{+})$-replication]{The baby monster Lie algebra and $\boldsymbol{(2{+})}$-replication}\label{BMandrep}

The moonshine module $V^\natural$ was constructed in \cite{FLM2} by Frenkel, Lepowsky and Meurman. This is a vertex operator algebra that has $\mathbb{M}$, the monster group, as symmetry group A vertex operator algebra is an intricate algebraic structure and we refer to \cite{LepLi} for the definition and the basics of its theory. The moonshine module has a grading $V^\natural= \bigoplus_{n\geq -1}V^\natural_{(n)}$ and its graded dimension $ \sum\limits_{n\geq -1}\big(\dim V^\natural_{(n)}\big)q^n$ is the $J$-function. Borcherds (\cite{B}) uses this vertex operator algebra to prove the moonshine conjectures in the following way. First, he shows that the McKay--Thompson series for $V^\natural$, i.e., $T_g= \sum\limits_{n\geq -1}\operatorname{Tr}\big(g|V^\natural_{(n)}\big)q^n$ are completely replicable functions. To do this, he uses~$V^\natural$ to build a generalized Kac--Moody algebra, the monster Lie algebra, whose twisted denominator identity is essentially the statement that the $n$-th replicate of a~$T_g$ is~$T_{g^n}$. Knowing that Hauptmoduls for genus-zero congruence groups are also completely replicable functions and that completely replicable functions satisfy some recurrence relations that determine a function from the first 5 coefficients of the function and its replicates, he was able to show that every McKay--Thompson series is indeed a Hauptmodul for some genus-zero congruence groups by just comparing the first few coefficients of the functions involved.

In \cite{H}, H\"ohn shows that there is a vertex algebra $W$ where $2\cdot\mathbb{B}$ acts as a symmetry group. This group is a central extension of $\mathbb{B}$, the baby monster group, and it arises as the centralizer of an element of class~$2A$ in~$\mathbb{M}$. This vertex operator algebra plays for $2\cdot\mathbb{B}$ the role that $V^\natural$ plays for the $\mathbb{M}$ and it was used to prove the generalized moonshine conjectures for the case of the baby monster, i.e., when $g$ in $T_{g,h}$ (see~\cite{N} for a precise statement of the generalized moonshine conjectures and what $T_{g,h}$ is) is an involution of type $2A$ in $\mathbb{M}$. In this section, we use $t$ to represent a (fixed) element in class~$2A$ in~$\mathbb{M}$. More precisely, what H\"ohn states in \cite{H} is the following.
If \begin{gather*}V^\natural(t)= \bigoplus_{n\geq -1}V^\natural_{(\frac{n}{2})}(t)\end{gather*} is the $t$-twisted module and $h$ is an element in the centralizer of $t$ in $\mathbb{M}$ then the McKay--Thompson series
\begin{gather*}T_{t,g}= \sum_{n\geq -1} \operatorname{Tr}\big(g|V^\natural_{\left(\frac{n}{2}\right)}(t)\big)q^n\end{gather*} is the Hauptmodul for some genus zero congruence subgroup.

We give a very brief sketch of the results from~\cite{H}. From the decompositions of $V^\natural=V^{00}\bigoplus V^{01}$ and $V^\natural(t)=V^{10}\bigoplus V^{11}$ of $+1$ and $-1$ eigenspaces for $t$, H\"ohn builds the vertex algebra $W$ mentioned above on which $2\cdot\mathbb{B}$ acts. Using this vertex algebra $W$ a Lie algebra $g_\mathbb{B}^\natural$, the baby monster Lie algebra, is constructed too. This is a $\frac{1}{2}\mathbb{Z}\times\frac{1}{2}\mathbb{Z}$-graded Lie algebra that has an action of $2\cdot\mathbb{B}$ in it that respects the grading and its $\big(\frac{m}{2},\frac{n}{2}\big)$ piece is isomorphic to $V^{[m,n]}_{(2mn)}$ ($[\cdot,\cdot]$ represents reduction ${\rm mod}\,2$). Also, $V^{10}$ is isomorphic to~$V^{01}$ as $2\cdot\mathbb{B}$-modules. H\"ohn shows that~$W$ is a generalized Kac--Moody algebra with twisted denominator formula
\begin{gather}
\sum_{m \in \mathbb{Z}} \operatorname{Tr}\big(g|V^{[m,1]}_{\left(\frac{m}{2}\right)}\big)p^\frac{m}{2}-\sum_{n \in \mathbb{Z}} \operatorname{Tr}\big(g|V^{[1,n]}_{\left(\frac{n}{2}\right)}\big)q^\frac{n}{2}\nonumber\\
\qquad{} =p^{-\frac{1}{2}} \exp\left(- \sum_{i>0}\sum_{\substack{m \in \mathbb{Z}^{+}\\n \in \mathbb{Z}}} \operatorname{Tr} \big(g^i|V^{[m,n]}_{\left(\frac{mn}{2}\right)}\big)\frac{p^\frac{im}{2}q^\frac{in}{2}}{i} \right).\label{twstd}
\end{gather}

We can now state our main result.

\begin{Theorem}\label{maintheorem}
The McKay--Thompson series $T_{t,g}(z)$, for $g\in2\cdot\mathbb{B}$, are completely \tp-replicable with replicates $T_{t,g}^{[n]}=T_{t,g^n}$ and $T_{t,g}^{[\sqrt{2}n]}=T_{1,g^{n}t}$.
\end{Theorem}
\begin{proof}

From the twisted denominator identity for the baby monster Lie algebra (\ref{twstd}) we have, after substituting $p^\frac{1}{2}$ for $p$ and $q^\frac{1}{2}$ for $q$,
\begin{gather*}
\sum_{m \in \mathbb{Z}} \operatorname{Tr}\big(g|V^{[m,1]}_{\left(\frac{m}{2}\right)}\big)p^m-\sum_{n \in \mathbb{Z}} \operatorname{Tr}\big(g|V^{[1,n]}_{\left(\frac{n}{2}\right)}\big)q^n =p^{-1} \exp\left(- \sum_{i>0}\sum_{\substack{m \in \mathbb{Z}^{+}\\n \in \mathbb{Z}}} \operatorname{Tr} \big(g^i|V^{[m,n]}_{\left(\frac{mn}{2}\right)}\big)\frac{p^{im}q^{in}}{i} \right).
\end{gather*}

The right side of this equation is $p^{-1}\exp(Z)$, where
\begin{gather*}
Z =- \sum_{i>0}\sum_{\substack{m \in \mathbb{Z}^{+}\\n \in \mathbb{Z}}} \operatorname{Tr} \big(g^i|V^{[m,n]}_{\left(\frac{mn}{2}\right)}\big)\frac{p^{im}q^{in}}{i}.
\end{gather*}

Because of the isomorphism between $V^{01}$ and $V^{10}$,
\begin{gather*}Z=- \sum_{i>0}
\sum_{\substack{m \in \mathbb{Z}^{+}\\n \in \mathbb{Z}\\ m \ \text{or} \ n \ \text{odd}}} \operatorname{Tr} \big(g^i|V^{[1,mn]}_{\left(\frac{mn}{2}\right)}\big)\frac{p^{im}q^{in}}{i}
- \sum_{i>0}\sum_{\substack{m \in \mathbb{Z}^{+}\\n \in \mathbb{Z}\\ m \ \text{and} \ n \ \text{even}}} \operatorname{Tr} \big(g^it|V^{00}_{\left(\frac{mn}{2}\right)}\big)\frac{p^{im}q^{in}}{i}
\\
\hphantom{Z}{} =- \sum_{i>0}\sum_{\substack{m \in \mathbb{Z}^{+}\\n \in \mathbb{Z}\\ m \ \text{or} \ n \ \text{odd}}} \operatorname{Tr} \big(g^i|V^{[1,mn]}_{\left(\frac{mn}{2}\right)}\big)\frac{p^{im}q^{in}}{i}\\
\hphantom{Z=}{}- \sum_{i>0}\sum_{\substack{m \in \mathbb{Z}^{+}\\n \in \mathbb{Z}\\ m \ \text{and} \ n \ \text{even}}} \left( \operatorname{Tr} \big(g^it|V^{\natural}_{\left(\frac{mn}{2}\right)}\big)- \operatorname{Tr} \big(g^it|V^{01}_{\left(\frac{mn}{2}\right)}\big)\right)\frac{p^{im}q^{in}}{i}
\\
\hphantom{Z}{} = - \sum_{i>0}\sum_{\substack{m \in \mathbb{Z}^{+}\\n \in \mathbb{Z}\\ m \ \text{or} \ n \ \text{odd}}} \operatorname{Tr}\big(g^i|V^{[1,mn]}_{\left(\frac{mn}{2}\right)}\big)\frac{p^{im}q^{in}}{i}\\
\hphantom{Z=}{} - \sum_{i>0}\sum_{\substack{m \in \mathbb{Z}^{+}\\n \in \mathbb{Z}\\ m \ \text{and} \ n \ \text{even}}} \left( \operatorname{Tr} \big(g^it|V^{\natural}_{\left(\frac{mn}{2}\right)}\big)+ \operatorname{Tr} \big(g^i|V^{01}_{\left(\frac{mn}{2}\right)}\big)\right)\frac{p^{im}q^{in}}{i}
\\
\hphantom{Z}{}=- \sum_{i>0}\sum_{\substack{m \in \mathbb{Z}^{+}\\n \in \mathbb{Z}}} \operatorname{Tr} \big(g^i|V^{[1,mn]}_{\left(\frac{mn}{2}\right)}\big)\frac{p^{im}q^{in}}{i}- \sum_{i>0}\sum_{\substack{m \in \mathbb{Z}^{+}\\n \in \mathbb{Z}}} \operatorname{Tr} \big(g^it|V^{\natural}_{(2mn)}\big)\frac{p^{2im}q^{2in}}{i}
\\
\hphantom{Z}{} =- \sum_{n=1}^{+\infty}\frac{1}{n}\left(\sum_{ad=n}d\cdot\sum_{k\in\mathbb{Z}} \operatorname{Tr} \big(g^a|V_{\left(\frac{kd}{2}\right)}^{[1,kd]}\big) q^{ak} + \sum_{\substack{ad=n\\ d \ \text{even} }}d\cdot\sum_{k\in\mathbb{Z}} \operatorname{Tr} \big(g^at|V_{\left(\frac{kd}{2}\right)}^{\natural}\big)q^{2ak}\right)p^{n}
\\
\hphantom{Z}{} =
- \sum_{n=1}^{+\infty}\frac{1}{n}\left(\sum_{\substack{ad=n\\0\leq b<d}} T_{t,g^a}\left(\frac{a\tau+b}{d}\right) +\sum_{\substack{ad=n \\ d \ \text{even} \\ 0 \leq b <d}} T_{1,g^at}\left(\frac{2a\tau+b}{d}\right)\right)p^{n}.
\end{gather*}

Since
\begin{gather*}T_{t,g}(p)-T_{t,g}(q)=p^{-1}\exp\left(- \sum_{n=1}^{+\infty}\frac{1}{n} P_{n}(T_{t,g}(q))p^n\right),\end{gather*}
where $P_n$ is the $n$-th Faber polynomial, we conclude that, for all $n\in \mathbb{N}$,
 \begin{gather*}\sum_{\substack{ad=n\\0\leq b<d}} T_{t,g^a}\left(\frac{a\tau+b}{d}\right) +\sum_{\substack{ad=n \\d \ \text{even} \\ 0 \leq b <d}} T_{1,g^at}\left(\frac{2a\tau+b}{d}\right)=P_{n}(T_{t,g}(q)),\end{gather*}
and we get that $T_{t,g}$ is \tp-replicable with \tp-replicates given as stated in the theorem.

By substituting $g$ by $g^n$ we see that $T_{t,g^n}$ is \tp-replicable with replicates $T_{t,g^n}^{[m]}=T_{t,g^{mn}}=T_{t,g}^{[mn]}$ and $T_{t,g^n}^{[\sqrt{2}m]}=T_{1,g^{mn}t}=T_{t,g}^{[\sqrt{2}nm]}$, i.e. $\big(T_{t,g}^{[n]}\big)^{[m]}=T_{t,g}^{[mn]}$ and $\big(T_{t,g}^{[n]}\big)^{[\sqrt{2}m]}=T_{t,g}^{[\sqrt{2}nm]}$. To complete the proof it remains to see that $T_{t,g}^{[\sqrt{2}n]}=T_{1,g^nt}$ is \tp-replicable with replicates $\big(T_{t,g}^{[\sqrt{2}n]}\big)^{[m]}=T_{1,g^{nm}t}$ and $\big(T_{t,g}^{[\sqrt{2}n]}\big)^{[\sqrt{2}m]}=T_{t,g^{2nm}}$, for $m\in\nat$. Equivalently, what we need to prove is that, for every $m\in \nat$,
\begin{gather*}
\sum_{\substack{ad=n\\0\leq b<d}} T_{1,g^{ma}t}\left(\frac{a\tau+b}{d}\right) +\sum_{\substack{ad=n \\ d \ \text{even} \\ 0 \leq b <d}} T_{t,g^{2ma}}\left(\frac{2a\tau+b}{d}\right)=P_{n}\big(T_{1,g^{m}t}(q)\big).
\end{gather*}

But since $T_{1,g^{m}t}$ is a monstrous function we know that
\begin{gather*}
\sum_{\substack{ad=n\\0\leq b<d}} T_{1,(g^{m}t)^{a}}\left(\frac{a\tau+b}{d}\right)=P_{n}\big(T_{1,g^{m}t}(q)\big).
\end{gather*}

But now,
\begin{gather*}\sum_{\substack{ad=n\\0\leq b<d}} T_{1,g^{ma}t}\left(\frac{a\tau+b}{d}\right)=n \sum_{k\in\nat}\left(\sum_{a|(n,k)}\frac{1}{a}\operatorname{Tr}\big(g^{ma}t|V^\natural_{\left(\frac{nk}{a^2}\right)}\big)\right)q^k,\\
\sum_{\substack{ad=n \\d \ \text{even} \\ 0 \leq b <d}} T_{t,g^{2ma}}\left(\frac{2a\tau+b}{d}\right)=n \sum_{k\in\nat}\left(\sum_{\substack{a|(n,k)\\ \frac{n}{a} \ \text{even}}}\frac{1}{a}\operatorname{Tr}\big(g^{2ma}|V^{[1,\frac{kn}{a^2}]}_{\left(\frac{nk}{2a^2}\right)}\big)\right)q^{2k},\\
\sum_{\substack{ad=n\\ 0\leq b<d}} T_{1,(g^{m}t)^{a}}\left(\frac{a\tau+b}{d}\right)=n \sum_{k\in\nat}\left(\sum_{a|(n,k)}\frac{1}{a}\operatorname{Tr}\big((g^mt)^a|V^\natural_{\left(\frac{nk}{a^2}\right)}\big)\right)q^k,
\end{gather*}
and what we have to show is
\begin{gather*}
 \sum_{k\in\nat}\sum_{a|(n,k)}\frac{1}{a}\operatorname{Tr}\big(g^{ma}t|V^\natural_{\left(\frac{nk}{a^2}\right)}\big)q^k+ \sum_{k\in\nat}\sum_{\substack{a|(n,k)\\ \frac{n}{a} \ \text{even}}}\frac{1}{a}\operatorname{Tr}\big(g^{2ma}|V^{[1,\frac{kn}{a^2}]}_{\left(\frac{nk}{2a^2}\right)}\big)q^{2k}\\
\qquad{} = \sum_{k\in\nat}\sum_{a|(n,k)}\frac{1}{a}\operatorname{Tr}\big((g^mt)^a|V^\natural_{\left(\frac{nk}{a^2}\right)}\big)q^k.
\end{gather*}

This means that for $k$ odd we have to show that
\begin{gather*}
 \sum_{a|(n,k)}\frac{1}{a}\operatorname{Tr}\big(g^{ma}t|V^\natural_{\left(\frac{nk}{a^2}\right)}\big)= \sum_{a|(n,k)}\frac{1}{a}\operatorname{Tr}\big((g^mt)^a|V^\natural_{\left(\frac{nk}{a^2}\right)}\big),
\end{gather*}
which is true because $t$ and $g$ commute and every $a$ in the sum, being a divisor of $k$, is odd too.

For $k=2k'$ even we have to show that
\begin{gather*}
 \sum_{a|(n,2k')}\frac{1}{a}\operatorname{Tr}\big(g^{ma}t|V^\natural_{\left(\frac{2nk'}{a^2}\right)}\big)+
 \sum_{\substack{a|(n,k')\\ \frac{n}{a} \ \text{even}}}\frac{1}{a}\operatorname{Tr}\big(g^{2ma}|V^{[1,\frac{k'n}{a^2}]}_{\left(\frac{nk'}{2a^2}\right)}\big)\\
\qquad{} = \sum_{a|(n,2k')}\frac{1}{a}\operatorname{Tr}\big((g^mt)^a|V^\natural_{\left(\frac{2nk'}{a^2}\right)}\big).
\end{gather*}

This identity is clearly true for $n$ odd and for $n=2n'$ even it becomes, because of the isomorphism between $V^{01}$ and $V^{10}$,
\begin{gather*}
 \sum_{a|2(n',k')}\frac{1}{a}\operatorname{Tr}\big(g^{ma}t|V^\natural_{\left(\frac{4n'k'}{a^2}\right)}\big)+
 \sum_{\substack{a|(2n',k')\\ \frac{2n'}{a} \ \text{even}}}\frac{1}{a}\operatorname{Tr}\big(g^{2ma}|V^{01}_{\left(\frac{n'k'}{a^2}\right)}\big)\\
 \qquad{} = \sum_{a|2(n',k')}\frac{1}{a}\operatorname{Tr}\big((g^mt)^a|V^\natural_{\left(\frac{4n'k'}{a^2}\right)}\big)
\end{gather*}
and this is now easy to prove
\begin{gather*}
 \sum_{a|2(n',k')}\frac{1}{a}\operatorname{Tr}\big(g^{ma}t|V^\natural_{\left(\frac{4n'k'}{a^2}\right)}\big)+
 \sum_{\substack{a|(2n',k')\\ \frac{2n'}{a} \ \text{even}}}\frac{1}{a}\operatorname{Tr}\big(g^{2ma}|V^{01}_{\left(\frac{n'k'}{a^2}\right)}\big)
\\
\qquad {}= \sum_{a|2(n',k')}\frac{1}{a}\operatorname{Tr}\big(g^{ma}t|V^\natural_{\left(\frac{4n'k'}{a^2}\right)}\big)+
 \sum_{a|(n',k')}\frac{1}{a}\operatorname{Tr}\big(g^{2ma}|V^{01}_{\left(\frac{4n'k'}{(2a)^2}\right)}\big)
\\
\qquad {}= \sum_{a|2(n',k')}\frac{1}{a}\operatorname{Tr}\big(g^{ma}t|V^\natural_{\left(\frac{4n'k'}{a^2}\right)}\big)
- \sum_{\substack{a|2(n',k')\\ a\ \text{even}}}\frac{2}{a}\operatorname{Tr}\big(g^{ma}t|V^{01}_{\left(\frac{4n'k'}{a^2}\right)}\big)
\\
\qquad {}= \sum_{\substack{a|2(n',k')}}\frac{1}{a}\operatorname{Tr}\big(g^{ma}t|V^{00}_{\left(\frac{4n'k'}{a^2}\right)}\big)
 + \sum_{\substack{a|2(n',k')\\ a \ \text{even}}}\frac{1}{a}\operatorname{Tr}\big(g^{ma}t|V^{01}_{\left(\frac{4n'k'}{a^2}\right)}\big)\\
\qquad \quad {} + \sum_{\substack{a|2(n',k')\\ a \ \text{odd}}}\frac{1}{a}\operatorname{Tr}\big(g^{ma}t|V^{01}_{\left(\frac{4n'k'}{a^2}\right)}\big)
- \sum_{\substack{a|2(n',k')\\ a\ \text{even}}}\frac{2}{a}\operatorname{Tr}\big(g^{ma}t|V^{01}_{\left(\frac{4n'k'}{a^2}\right)}\big)
\\
\qquad {}= \sum_{\substack{a|2(n',k')}}\frac{1}{a}\operatorname{Tr}\big((g^{m}t)^{a}|V^{00}_{\left(\frac{4n'k'}{a^2}\right)}\big)
+ \sum_{\substack{a|2(n',k')\\ a \ \text{even}}}\frac{1}{a}\operatorname{Tr}\big((g^{m}t)^a|V^{01}_{\left(\frac{4n'k'}{a^2}\right)}\big)\\
\qquad \quad {} + \sum_{\substack{a|2(n',k')\\ a \ \text{odd}}}\frac{1}{a}\operatorname{Tr}\big((g^{m}t)^a|V^{01}_{\left(\frac{4n'k'}{a^2}\right)}\big)= \sum_{a|2(n',k')}\frac{1}{a}\operatorname{Tr}\big((g^mt)^a|V^\natural_{\left(\frac{4n'k'}{a^2}\right)}\big),
\end{gather*}
and the theorem is proven.
\end{proof}

We can now use Theorem~\ref{maintheorem} to reprove H\"ohn's result

\begin{Theorem}\label{Hoenstheorem}
The McKay--Thompson series $T_{t,g}(z)$, for $g\in2\cdot\mathbb{B}$, are the $q$-expansions of Hauptmoduls.
\end{Theorem}

\begin{proof}We know that the McKay--Thompson series $T_{t,g}$ are completely \tp-replicable and consequently satisfy the recurrence relations from Theorem~\ref{MahlerRec}. We know that $T^{[n\sqrt{2}]}_{t,g}=T_{1,tg^{n}}$ is a monstrous function and therefore its coefficients are known once we know in what class in the monster the element~$tg$ is, for every $g\in 2\cdot\mathbb{B}$. This can be done with GAP~\cite{GAP4}. Hence, the first five coefficients of every $T_{t,g}$ determine all the coefficients of the $T_{t,g}$ completely. From Section~\ref{CRHauptmoduls} we also have some completely \tp-replicability results for some Hauptmoduls and thus these Hauptmoduls satisfy the same recurrence relations from Theorem~\ref{MahlerRec}. To prove that every McKay--Thompson series is a Hauptmodul it would be enough to compare, for every $g\in 2\cdot\mathbb{B}$, the first five coefficients of $T_{t,g}, T_{1,tg}, T_{t,g^2}, T_{1,tg^2},\ldots$ with those of $f, f^{[\sqrt{2}]}, f^{[2]}, f^{[2\sqrt{2}]},\ldots$, respectively, for some Hauptmodul $f$ in Tables~\ref{Table1} and~\ref{Table2}.

However, not all McKay--Thompson series correspond to Hauptmoduls listed in Tables~\ref{Table1} and~\ref{Table2} and, because of that, this method for proving that the McKay--Thompson series are Hauptmoduls works for all 247 classes in $2\cdot \mathbb{B}$ with 13 exceptions. This happens because the Hauptmodul associated to each of these 13 classes is neither a completely replicable function nor a dash (see \cite{FMN} for the definition of the dash operator) of a completely replicable function and our Tables~\ref{Table1} and~\ref{Table2} only contain such functions.

\begin{table}[t!]
\centering
\scalebox{0.8}{
\begin{tabular}{|c|c|c||c|c|c|}
\hline
\mbox{Class of $g$ (of $g^2$)} & $T_{t,g}$ & $T^{[\sqrt{2}]}_{t,g}=T_{1,gt}$ & \mbox{Class of $g$ (of $g^2$)} & $T_{t,g}$ & $T^{[\sqrt{2}]}_{t,g}=T_{1,gt}$ \\\hline
$1a\:(1a) $ & $ 2A $ & $ 2A $ & $ 2a\:(1a) $ & $ 4\sim b $ & $ 1A $ \\\hline
$2b\:(1a) $ & $ 2a $ & $ 2A $ & $ 2c\:(1a) $ & $ 4A $ & $ 2B $ \\\hline
$2d\:(1a) $ & $ 2B $ & $ 2A $ & $ 2e\:(1a) $ & $ 4C $ & $ 2B $ \\\hline
$3a\:(3a) $ & $ 6A $ & $ 6A $ & $ 3b\:(3b) $ & $ 6D $ & $ 6D $ \\\hline
$4a\:(2a) $ & $ 8\sim b $ & $ 4B $ & $ 4b\:(2d) $ & $ 4a $ & $ 4A $ \\\hline
$4c\:(2d) $ & $ 4B $ & $ 4A $ & $ 4d\:(2d) $ & $ 4C $ & $ 4C $ \\\hline
$4e\:(2c) $ & $ 8a $ & $ 4B $ & $ 4f\:(2e) $ & $ 8A $ & $ 4C $ \\\hline
$4g\:(2e) $ & $ 8\sim d $ & $ 4A $ & $ 4h\:(2d) $ & $ 4D $ & $ 4C $ \\\hline
$4i\:(2c) $ & $ 8B $ & $ 4B $ & $ 4j\:(2e) $ & $ 8E $ & $ 4C $ \\\hline
$4k\:(2e) $ & $ 8D $ & $ 4D $ & $ 5a\:(5a) $ & $ 10A $ & $ 10A $ \\\hline
$5b\:(5b) $ & $ 10C $ & $ 10C $ & $ 6a\:(3a) $ & $ 12\sim d $ & $ 3A $ \\\hline
$6b\:(3b) $ & $ 12\sim f $ & $ 3B $ & $ 6c\:(3a) $ & $ 6a $ & $ 6A $ \\\hline
$6d\:(3a) $ & $ 6b $ & $ 6A $ & $ 6e\:(3a) $ & $ 12A $ & $ 6C $ \\\hline
$6f\:(3a) $ & $ 6C $ & $ 6A $ & $ 6g\:(3b) $ & $ 6c $ & $ 6D $ \\\hline
$6h\:(3a) $ & $ 12c $ & $ 6C $ & $ 6i\:(3b) $ & $ 12B $ & $ 6E $ \\\hline
$6j\:(3b) $ & $ 6E $ & $ 6D $ & $ 6k\:(3a) $ & $ 12E $ & $ 6C $ \\\hline
$6l\:(3b) $ & $ 12H $ & $ 6E $ & $ 6m\:(3b) $ & $ 12\sim h $ & $ 6B $ \\\hline
$6n\:(3b) $ & $ 12I $ & $ 6E $ & $ 7a\:(7a) $ & $ 14A $ & $ 14A $ \\\hline
$8a\:(4a) $ & $ 16\sim a $ & $ 8C $ & $ 8b\:(4d) $ & $ 8a $ & $ 8A $ \\\hline
$8c\:(4c) $ & $ 8b $ & $ 8B $ & $ 8d\:(4c) $ & $ 8c $ & $ 8B $ \\\hline
$8e\:(4d) $ & $ 8B $ & $ 8A $ & $ 8f\:(4c) $ & $ 8C $ & $ 8B $ \\\hline
$8g\:(4d) $ & $ 8D $ & $ 8E $ & $ 8h\:(4g) $ & $ 16\sim d $ & $ 8D $ \\\hline
$8i\:(4d) $ & $ 8E $ & $ 8E $ & $ 8j\:(4h) $ & $ 8F $ & $ 8D $ \\\hline
$8k\:(4f) $ & $ 16A $ & $ 8B $ & $ 8l\:(4j) $ & $ 16C $ & $ 8E $ \\\hline
$8m\:(4j) $ & $ 16\sim e $ & $ 8A $ & $ 8n\:(4i) $ & $ 16a $ & $ 8C $ \\\hline
$8o\:(4j) $ & $ 16B $ & $ 8E $ & $ 8p\:(4k) $ & $ 16d $ & $ 8F $ \\\hline
$9a\:(9a) $ & $ 18A $ & $ 18A $ & $ 9b\:(9b) $ & $ 18B $ & $ 18B $ \\\hline
$10a\:(5a) $ & $ 20\sim c $ & $ 5A $ & $ 10b\:(5b) $ & $ 20\sim d $ & $ 5B $ \\\hline
$10c\:(5a) $ & $ 10a $ & $ 10A $ & $ 10d\:(5a) $ & $ 20A $ & $ 10B $ \\\hline
$10e\:(5a) $ & $ 10B $ & $ 10A $ & $ 10f\:(5b) $ & $ 20C $ & $ 10E $ \\\hline
$10g\:(5b) $ & $ 10E $ & $ 10C $ & $ 10h\:(5a) $ & $ 20d $ & $ 10B $ \\\hline
$10i\:(5b) $ & $ 20F $ & $ 10E $ & $ 10j\:(5b) $ & $ 20\sim g $ & $ 10D $ \\\hline
$11a\:(11a) $ & $ 22A $ & $ 22A $ & $ 12a\:(6a) $ & $ 24\sim f $ & $ 12C $ \\\hline
$12b\:(6b) $ & $ 24\sim h $ & $ 12G $ & $ 12c\:(6f) $ & $ 12a $ & $ 12A $ \\\hline
$12d\:(6j) $ & $ 12G $ & $ 12B $ & $ 12e\:(6f) $ & $ 12b $ & $ 12A $ \\\hline
$12f\:(6f) $ & $ 12C $ & $ 12A $ & $ 12g\:(6e) $ & $ 24a $ & $ 12C $ \\\hline
$12h\:(6m) $ & $ 24\sim j $ & $ 12H $ & $ 12i\:(6m) $ & $ 24\sim k $ & $ 12B $ \\\hline
$12j\:(6f) $ & $ 12E $ & $ 12E $ & $ 12k\:(6e) $ & $ 24b $ & $ 12C $ \\\hline
$12l\:(6f) $ & $ 12d $ & $ 12E $ & $ 12m\:(6k) $ & $ 24B $ & $ 12E $ \\\hline
$12n\:(6k) $ & $ 24\sim m $ & $ 12A $ & $ 12o\:(6i) $ & $ 24c $ & $ 12G $ \\\hline
$12p\:(6e) $ & $ 24A $ & $ 12C $ & $ 12q\:(6j) $ & $ 12I $ & $ 12I $ \\\hline
$12r\:(6n) $ & $ 24C $ & $ 12I $ & $ 12s\:(6n) $ & $ 24\sim o $ & $ 12B $ \\\hline
$12t\:(6j) $ & $ 12F $ & $ 12H $ & $ 12u\:(6k) $ & $ 24h $ & $ 12E $ \\\hline
$12v\:(6m) $ & $ 24\sim q $ & $ 12I $ & $ 12w\:(6l) $ & $ 24H $ & $ 12F $ \\\hline
$12x\:(6n) $ & $ 24I $ & $ 12I $ & $ 12y\:(6n) $ & $ 24\sim r $ & $ 12H $ \\\hline
$13a\:(13a) $ & $ 26A $ & $ 26A $ & $ 14a\:(7a) $ & $ 28\sim c $ & $ 7A $ \\\hline
$14b\:(7a) $ & $ 14a $ & $ 14A $ & $ 14c\:(7a) $ & $ 14c $ & $ 14A $ \\\hline
$14d\:(7a) $ & $ 28B $ & $ 14B $ & $ 14e\:(7a) $ & $ 14B $ & $ 14A $ \\\hline
$14f\:(7a) $ & $ 28C $ & $ 14B $ & $ 15a\:(15a) $ & $ 30B $ & $ 30B $ \\\hline
$15b\:(15b) $ & $ 30F $ & $ 30F $ & $ 16a\:(8e) $ & $ 16b $ & $ 16A $ \\\hline
$16b\:(8e) $ & $ 16c $ & $ 16A $ & $ 16c\:(8i) $ & $ 16B $ & $ 16B $ \\\hline
$16d\:(8i) $ & $ 16A $ & $ 16C $ & $ 16e\:(8e) $ & $ 16a $ & $ 16A $ \\\hline
$16f\:(8i) $ & $ 16A $ & $ 16C $ & $ 16g\:(8l) $ & $ 32B $ & $ 16A $ \\\hline
$16h\:(8o) $ & $ 32A $ & $ 16B $ & $ 16i\:(8o) $ & $ 32\sim e $ & $ 16C $ \\\hline
$17a\:(17a) $ & $ 34A $ & $ 34A $ & $ 18a\:(9a) $ & $ 36\sim h $ & $ 9B $ \\\hline
$18b\:(9b) $ & $ 36\sim e $ & $ 9A $ & $ 18c\:(9b) $ & $ 18c $ & $ 18B $ \\\hline
$18d\:(9b) $ & $ 18c $ & $ 18B $ & $ 18e\:(9b) $ & $ 36A $ & $ 18C $ \\\hline
\end{tabular}}
\end{table}

\begin{table}[t!]
\centering
\scalebox{0.8}{
\begin{tabular}{|c|c|c||c|c|c|}
\hline
\mbox{Class of $g$ (of $g^2$)} & $T_{t,g}$ & $T^{[\sqrt{2}]}_{t,g}=T_{1,gt}$ & \mbox{Class of $g$ (of $g^2$)} & $T_{t,g}$ & $T^{[\sqrt{2}]}_{t,g}=T_{1,gt}$ \\\hline
$18f\:(9b) $ & $ 18C $ & $ 18B $ & $ 18g\:(9b) $ & $ 36f $ & $ 18C $ \\\hline
$18h\:(9a) $ & $ 36D $ & $ 18D $ & $ 18i\:(9a) $ & $ 36\sim q $ & $ 18E $ \\\hline
$19a\:(19a) $ & $ 38A $ & $ 38A $ & $ 20a\:(10a) $ & $ 40\sim c $ & $ 20B $ \\\hline
$20b\:(10e) $ & $ 20a $ & $ 20A $ & $ 20c\:(10g) $ & $ 20c $ & $ 20C $ \\\hline
$20d\:(10e) $ & $ 20b $ & $ 20A $ & $ 20e\:(10e) $ & $ 20B $ & $ 20A $ \\\hline
$20f\:(10d) $ & $ 40a $ & $ 20B $ & $ 20g\:(10d) $ & $ 40B $ & $ 20B $ \\\hline
$20h\:(10j) $ & $ 40\sim h $ & $ 20F $ & $ 20i\:(10j) $ & $ 40\sim i $ & $ 20C $ \\\hline
$20j\:(10g) $ & $ 20E $ & $ 20F $ & $ 20k\:(10i) $ & $ 40C $ & $ 20E $ \\\hline
$21a\:(21a) $ & $ 42A $ & $ 42A $ & $ 22a\:(11a) $ & $ 44\sim b $ & $ 11A $ \\\hline
$22b\:(11a) $ & $ 22a $ & $ 22A $ & $ 22c\:(11a) $ & $ 22a $ & $ 22A $ \\\hline
$22d\:(11a) $ & $ 44A $ & $ 22B $ & $ 22e\:(11a) $ & $ 22B $ & $ 22A $ \\\hline
$23a\:(23a) $ & $ 46C $ & $ 46C $ & $ 23b\:(23b) $ & $ 46C $ & $ 46C $ \\\hline
$24a\:(12b) $ & $ 48\sim c $ & $ 24G $ & $ 24b\:(12f) $ & $ 24d $ & $ 24A $ \\\hline
$24c\:(12f) $ & $ 24e $ & $ 24A $ & $ 24d\:(12f) $ & $ 24g $ & $ 24A $ \\\hline
$24e\:(12f) $ & $ 24f $ & $ 24A $ & $ 24f\:(12j) $ & $ 24b $ & $ 24B $ \\\hline
$24g\:(12q) $ & $ 24c $ & $ 24C $ & $ 24h\:(12j) $ & $ 24A $ & $ 24B $ \\\hline
$24i\:(12i) $ & $ 48\sim h $ & $ 24H $ & $ 24j\:(12n) $ & $ 48\sim i $ & $ 24D $ \\\hline
$24k\:(12m) $ & $ 48A $ & $ 24A $ & $ 24l\:(12q) $ & $ 24H $ & $ 24I $ \\\hline
$24m\:(12v) $ & $ 48\sim j $ & $ 24I $ & $ 24n\:(12v) $ & $ 48\sim k $ & $ 24C $ \\\hline
$24o\:(12t) $ & $ 24F $ & $ 24H $ & $ 24p\:(12w) $ & $ 48h $ & $ 24F $ \\\hline
$25a\:(25a) $ & $ 50A $ & $ 50A $ & $ 26a\:(13a) $ & $ 52\sim c $ & $ 13A $ \\\hline
$26b\:(13a) $ & $ 26a $ & $ 26A $ & $ 27a\:(27a) $ & $ 54A $ & $ 54A $ \\\hline
$28a\:(14a) $ & $ 56\sim d $ & $ 28A $ & $ 28b\:(14e) $ & $ 28A $ & $ 28B $ \\\hline
$28c\:(14e) $ & $ 28C $ & $ 28C $ & $ 28d\:(14e) $ & $ 28a $ & $ 28B $ \\\hline
$28e\:(14d) $ & $ 56a $ & $ 28A $ & $ 28f\:(14f) $ & $ 56A $ & $ 28C $ \\\hline
$28g\:(14f) $ & $ 56\sim g $ & $ 28B $ & $ 30a\:(15a) $ & $ 60\sim c $ & $ 15A $ \\\hline
$30b\:(15b) $ & $ 60\sim l $ & $ 15C $ & $ 30c\:(15a) $ & $ 30a $ & $ 30B $ \\\hline
$30d\:(15a) $ & $ 30d $ & $ 30B $ & $ 30e\:(15a) $ & $ 60B $ & $ 30C $ \\\hline
$30f\:(15a) $ & $ 30C $ & $ 30B $ & $ 30g\:(15a) $ & $ 60a $ & $ 30C $ \\\hline
$30h\:(15b) $ & $ 60D $ & $ 30G $ & $ 30i\:(15b) $ & $ 60\sim m $ & $ 30A $ \\\hline
$30j\:(15b) $ & $ 60C $ & $ 30G $ & $ 30k\:(15b) $ & $ 30G $ & $ 30F $ \\\hline
$30l\:(15b) $ & $ 60C $ & $ 30G $ & $ 30m\:(15b) $ & $ 30G $ & $ 30F $ \\\hline
$31a\:(31a) $ & $ 62A $ & $ 62A $ & $ 31b\:(31b) $ & $ 62A $ & $ 62A $ \\\hline
$32a\:(16c) $ & $ 32B $ & $ 32A $ & $ 32b\:(16c) $ & $ 32B $ & $ 32A $ \\\hline
$32c\:(16d) $ & $ 32b $ & $ 32B $ & $ 32d\:(16d) $ & $ 32b $ & $ 32B $ \\\hline
$33a\:(33a) $ & $ 66A $ & $ 66A $ & $ 34a\:(17a) $ & $ 68\sim b $ & $ 17A $ \\\hline
$34b\:(17a) $ & $ 34a $ & $ 34A $ & $ 34c\:(17a) $ & $ 34a $ & $ 34A $ \\\hline
$35a\:(35a) $ & $ 70A $ & $ 70A $ & $ 36a\:(18b) $ & $ 72\sim c $ & $ 36C $ \\\hline
$36b\:(18f) $ & $ 36C $ & $ 36A $ & $ 36c\:(18e) $ & $ 72a $ & $ 36C $ \\\hline
$36d\:(18i) $ & $ 72\sim p $ & $ 36D $ & $ 36e\:(18i) $ & $ 72\sim q $ & $ 36B $ \\\hline
$38a\:(19a) $ & $ 76\sim b $ & $ 19A $ & $ 38b\:(19a) $ & $ 38a $ & $ 38A $ \\\hline
$38c\:(19a) $ & $ 38a $ & $ 38A $ & $ 39a\:(39a) $ & $ 78A $ & $ 78A $ \\\hline
$40a\:(20a) $ & $ 80\sim a $ & $ 40A $ & $ 40b\:(20e) $ & $ 40b $ & $ 40B $ \\\hline
$40c\:(20e) $ & $ 40c $ & $ 40B $ & $ 40d\:(20e) $ & $ 40A $ & $ 40B $ \\\hline
$40e\:(20g) $ & $ 80a $ & $ 40A $ & $ 40f\:(20i) $ & $ 80\sim e $ & $ 40C $ \\\hline
$40g\:(20i) $ & $ 80\sim e $ & $ 40C $ & $ 42a\:(21a) $ & $ 84\sim e $ & $ 21A $ \\\hline
$42b\:(21a) $ & $ 42a $ & $ 42A $ & $ 42c\:(21a) $ & $ 42b $ & $ 42A $ \\\hline
$44a\:(22e) $ & $ 44c $ & $ 44A $ & $ 44b\:(22e) $ & $ 44c $ & $ 44A $ \\\hline
$46a\:(23a) $ & $ 92\sim b $ & $ 23A $ & $ 46b\:(23b) $ & $ 92\sim b $ & $ 23A $ \\\hline
$46c\:(23c) $ & $ 92A $ & $ 46A $ & $ 46d\:(23a) $ & $ 46A $ & $ 46C $ \\\hline
$46e\:(23b) $ & $ 92A $ & $ 46A $ & $ 46f\:(23b) $ & $ 46A $ & $ 46C $ \\\hline
$47a\:(47a) $ & $ 94A $ & $ 94A $ & $ 47b\:(47b) $ & $ 94A $ & $ 94A $ \\\hline
$48a\:(24h) $ & $ 48a $ & $ 48A $ & $ 48b\:(24h) $ & $ 48b $ & $ 48A $ \\\hline
$50a\:(25a) $ & $ 100\sim c $ & $ 25A $ & $ 52a\:(26a) $ & $ 104\sim c $ & $ 52A $ \\\hline
$54a\:(27a) $ & $ 108\sim g $ & $ 27A $ & $ 55a\:(55a) $ & $ 110A $ & $ 110A $ \\\hline
$56a\:(28c) $ & $ 56a $ & $ 56A $ & $ 56b\:(28c) $ & $ 56a $ & $ 56A $ \\\hline
$60a\:(30a) $ & $ 120\sim d $ & $ 60A $ & $ 60b\:(30f) $ & $ 60b $ & $ 60B $ \\\hline
\end{tabular}}
\end{table}

\begin{table}[t!]
\centering
\scalebox{0.8}{
\begin{tabular}{|c|c|c||c|c|c|}
\hline
\mbox{Class of $g$ (of $g^2$)} & $T_{t,g}$ & $T^{[\sqrt{2}]}_{t,g}=T_{1,gt}$ & \mbox{Class of $g$ (of $g^2$)} & $T_{t,g}$ & $T^{[\sqrt{2}]}_{t,g}=T_{1,gt}$ \\\hline
$60c\:(30e) $ & $ 120a $ & $ 60A $ & $ 60d\:(30i) $ & $ 120\sim g $ & $ 60D $ \\\hline
$60e\:(30i) $ & $ 120 \sim h $ & $ 60C $ & $ 62a\:(31a) $ & $ 124\sim b $ & $ 31A $ \\\hline
$62b\:(31b) $ & $ 124\sim b $ & $ 31A $ & $ 66a\:(33a) $ & $ 132\sim c $ & $ 33B $ \\\hline
$66b\:(33a) $ & $ 66a $ & $ 66A $ & $ 66c\:(33a) $ & $ 66a $ & $ 66A $ \\\hline
$68a\:(34a) $ & $ 136\sim b $ & $ 68A $ & $ 70a\:(35a) $ & $ 140\sim b $ & $ 35A $ \\\hline
$70b\:(35a) $ & $ 70a $ & $ 70A $ & $ 70c\:(35a) $ & $ 70a $ & $ 70A $ \\\hline
$78a\:(39a) $ & $ 156\sim d $ & $ 39A $ & $ 84a\:(42a) $ & $ 168\sim c $ & $ 84A $ \\\hline
$94a\:(47a) $ & $ 188\sim b $ & $ 47A $ & $ 94b\:(47a) $ & $ 188\sim b $ & $ 47A $ \\\hline
$104a\:(52a) $ & $ 208\sim a $ & $ 104A $ & $ 104b\:(52a) $ & $ 208\sim a $ & $ 104A $ \\\hline
$110a\:(55a) $ & $ 220\sim b $ & $ 55A$ \\\cline{1-3}
\end{tabular}}
\caption{Classes in $2\cdot \mathbb{B}$ and corresponding McKay--Thompson series together with their $\big[\sqrt{2}\big]$-replicates.}\label{Table3}
\end{table}

For those $247-13=234$ classes covered by Tables~\ref{Table1} and~\ref{Table2} we use the decomposition of the first five head characters given in~\cite{H}:
\begin{itemize}\itemsep=0pt
 \item[1)] $H_1=\chi_1+\chi_2$,
 \item[2)] $H_2=\chi_{185}$,
 \item[3)] $H_3=2\chi_1+\chi_2+\chi_3+\chi_4$,
 \item[4)] $H_4=2\chi_{185}+\chi_{186}$,
 \item[5)] $H_5=3\chi_1+3\chi_2+2\chi_3+\chi_4+\chi_6+\chi_7$
\end{itemize}
to make the comparison and we obtain a proof of H\"ohn's results which is analogous to Borcherds proof of the original moonshine conjectures.

For the remaining 13 classes (their names are, using GAP notation: $12h$, $12i$, $20h$, $20i$, $24i$, $24m$, $24n$, $36d$, $36e$, $40f$, $40g$, $60d$, $60e$) we use again the recurrence relations from Theorem~\ref{MahlerRec} to find the first 23 coefficients of their McKay--Thompson
series. Since we know that a replicable function (we recall from Remark~\ref{equivrep2Arep} that a \tp-replicable function is replicable) is completely determined by its first 23 coefficients, a simple check among the power series expansions of the 616 Hauptmoduls for genus zero groups with rational integer coefficients allows us to match every such class with some Hauptmodul (see references \cite{ChCm, FMN,N3}
 for details of these functions and the corresponding groups). This is analogous to H\"ohn's proof but the recurrence relations from Theorem \ref{MahlerRec} simplify the computations greatly.
 \end{proof}

A list of $2\cdot \mathbb{B}$ classes with their corresponding McKay--Thompson series and $\left[\sqrt{2}\right]$-replicates is given in Table~\ref{Table3} (see~\cite{N3} for the labelling of the Hauptmoduls).

\section{Comments and questions}\label{commentsandquestions}

The results of the previous sections raise some natural questions. For example, to what extent can these results be extended to other groups arising as centralizers of elements in classes~$p+$ or~$p-$, for a prime~$p$, in the monster?

For groups of type $p+$ there is an obvious generalization of \tp-replication to $(p+)$-replica\-tion. Computationally we have found the $(p+)$-replicates of all of the rational replicable functions in Norton's list of 616 such functions~\cite{N3}. In all the cases we have checked, this replication structure respects the power map structure in the group arising as a centralizer of the element of class $p+$ in the monster. However, even if we could generalize all the results of the current paper, this would not be sufficient to give a separate proof akin to Borcherds' monstrous proof for each of these groups. In the first place, an explicit construction of each VOA-module would be required. However, more importantly, for $p>2$ there are Hauptmodul which have some irrational integer coefficients. There are two such Hauptmodul for $3\cdot Fi_{24}$. Although it is possible to write down the identities which are required in these cases, we have not found a~natural way to do so. The crux of the problem is that for the monster and $2\cdot\mathbb{B}$ a $p$-th replicate contains information about the $p$th power of an element~$g$ even when~$p$ divides the order of~$g$ and this information is contained in the essentially combinatorial definition of the $p$-replicate. No information about $g$ or the underlying action of~$g$ on the relevant VOA-module is needed. But when the corresponding Hauptmodul has irrational coefficients we have not found a way to extract this information (other than to look at the underlying VOA-module structure, of course). A solution to this problem for $3\cdot Fi_{24}$ would be very interesting. The classes in question are~$18m$ and~$18p$ and the Hauptmoduls are $54\sim c$ and $54\sim b$~\cite{N3}.

For groups of type $p-$ the situation is more complicated and we have found no natural way to choose an operator in the Hecke algebra that applied to the Hauptmodul for $p-$ yields the Faber polynomial while at the same time respecting the power map structure in corresponding group. This seems to be due to the fact that~$p-$, unlike~$p+$, has more than one cusp. A solution to this problem in the simplest case $2-$ would be very interesting.

We should also mention that complete replicability plays a key role in the work of Martin~\cite{Martin} on modularity of ``$j$-final'' functions and also in the work of Kozlov~\cite{Koz} and Cummins and Gannon~\cite{CG} on modularity from formal modular equations. The aim of these papers was to see if the identities implied by complete replicability imply the modularity of Thompson--McKay series thus avoiding the part of Borcherds' proof which shows modularity by recurrence relations and a comparison of initial terms. This approach was used by Carnahan in his work on generalized moonshine. We believe it would be of interest to extend these results to complete \tp-replicability, but we have not done so here.

\newpage
\subsection*{Acknowledgements}

The authors would like to thank the referees for their valuable comments and suggestions to improve the quality of the paper. The second author acknowledges financial assistance from the Center for Mathematics of the University of Coimbra—UID/MAT/00324/2013, funded by the Portuguese Government through FCT/MCTES and co-funded by the European Regional Development Fund through the Partnership Agreement PT2020.

\pdfbookmark[1]{References}{ref}
\LastPageEnding


\begin{thebibliography}{99}
\footnotesize\itemsep=-1pt

\bibitem{ACMS}
Alexander D., Cummins C., McKay J., Simons C., Completely replicable functions,
 in Groups, Combinatorics \& Geometry ({D}urham, 1990), \href{https://doi.org/10.1017/CBO9780511629259.010}{\textit{London Math.
 Soc. Lecture Note Ser.}}, Vol.~165, Cambridge University Press, Cambridge,
 1992, 87--98.

\bibitem{B2}
Borcherds R.E., Vertex algebras, {K}ac--{M}oody algebras, and the {M}onster,
 \href{https://doi.org/10.1073/pnas.83.10.3068}{\textit{Proc. Nat. Acad. Sci. USA}} \textbf{83} (1986), 3068--3071.

\bibitem{B3}
Borcherds R.E., Generalized {K}ac--{M}oody algebras, \href{https://doi.org/10.1016/0021-8693(88)90275-X}{\textit{J.~Algebra}}
 \textbf{115} (1988), 501--512.

\bibitem{B}
Borcherds R.E., Monstrous moonshine and monstrous {L}ie superalgebras,
 \href{https://doi.org/10.1007/BF01232032}{\textit{Invent. Math.}} \textbf{109} (1992), 405--444.

\bibitem{carn}
Carnahan S., Generalized moonshine~{I}: genus-zero functions, \href{https://doi.org/10.2140/ant.2010.4.649}{\textit{Algebra
 Number Theory}} \textbf{4} (2010), 649--679, \href{https://arxiv.org/abs/0812.3440}{arXiv:0812.3440}.

\bibitem{carn2}
Carnahan S., Generalized moonshine~{II}: {B}orcherds products, \href{https://doi.org/10.1215/00127094-1548416}{\textit{Duke
 Math.~J.}} \textbf{161} (2012), 893--950, \href{https://arxiv.org/abs/0908.4223}{arXiv:0908.4223}.

\bibitem{carn4}
Carnahan S., Generalized moonshine~{IV}: monstrous Lie algebras,
 \href{https://arxiv.org/abs/1208.6254}{arXiv:1208.6254}.

\bibitem{CN}
Conway J.H., Norton S.P., Monstrous moonshine, \href{https://doi.org/10.1112/blms/11.3.308}{\textit{Bull. London Math. Soc.}}
 \textbf{11} (1979), 308--339.

\bibitem{ChCm}
Cummins C.J., Congruence subgroups of groups commensurable with {${\rm
 PSL}(2,{\mathbb Z})$} of genus~0 and~1, \href{https://doi.org/10.1080/10586458.2004.10504547}{\textit{Experiment. Math.}}
 \textbf{13} (2004), 361--382.

\bibitem{CG}
Cummins C.J., Gannon T., Modular equations and the genus zero property of
 moonshine functions, \href{https://doi.org/10.1007/s002220050167}{\textit{Invent. Math.}} \textbf{129} (1997), 413--443.

\bibitem{CuN}
Cummins C.J., Norton S.P., Rational {H}auptmoduls are replicable,
 \href{https://doi.org/10.4153/CJM-1995-061-1}{\textit{Canad.~J. Math.}} \textbf{47} (1995), 1201--1218.

\bibitem{DSh}
Diamond F., Shurman J., A first course in modular forms, \href{https://doi.org/10.1007/978-0-387-27226-9}{\textit{Graduate Texts
 in Mathematics}}, Vol.~228, Springer-Verlag, New York, 2005.

\bibitem{Fer}
Ferenbaugh C.R., On the modular functions involved in ``{M}onstrous
 {M}oonshine'', Ph.D.~Thesis, Princeton University, 1992.

\bibitem{FMN}
Ford D., McKay J., Norton S., More on replicable functions, \href{https://doi.org/10.1080/00927879408825127}{\textit{Comm.
 Algebra}} \textbf{22} (1994), 5175--5193.

\bibitem{FLM2}
Frenkel I.B., Lepowsky J., Meurman A., A natural representation of the
 {F}ischer--{G}riess {M}onster with the modular function {$J$} as character,
 \href{https://doi.org/10.1073/pnas.81.10.3256}{\textit{Proc. Nat. Acad. Sci. USA}} \textbf{81} (1984), 3256--3260.

\bibitem{FLM}
Frenkel I.B., Lepowsky J., Meurman A., Vertex operator algebras and the
 {M}onster, \textit{Pure and Applied Mathematics}, Vol.~134, Academic Press,
 Inc., Boston, MA, 1988.

\bibitem{GAP4}
{GAP} -- {G}roups, {A}lgorithms, and {P}rogramming, Version~4.8.7,
 2017, \url{http://www.gap-system.org/}.

\bibitem{H}
H\"ohn G., Generalized moonshine for the baby monster, {P}reprint, 2003.

\bibitem{Koz}
Kozlov D.N., On functions satisfying modular equations for infinitely many
 primes, \href{https://doi.org/10.4153/CJM-1999-045-x}{\textit{Canad.~J. Math.}} \textbf{51} (1999), 1020--1034.

\bibitem{LepLi}
Lepowsky J., Li H., Introduction to vertex operator algebras and their
 representations, \href{https://doi.org/10.1007/978-0-8176-8186-9}{\textit{Progress in Mathematics}}, Vol.~227, Birkh\"auser
 Boston, Inc., Boston, MA, 2004.

\bibitem{Martin}
Martin Y., On modular invariance of completely replicable functions, in
 Moonshine, the {M}onster, and Related Topics ({S}outh {H}adley, {MA}, 1994),
 \href{https://doi.org/10.1090/conm/193/02375}{\textit{Contemp. Math.}}, Vol.~193, Amer. Math. Soc., Providence, RI, 1996,
 263--286.

\bibitem{N2}
Norton S.P., More on moonshine, in Computational Group Theory ({D}urham, 1982),
 Academic Press, London, 1984, 185--193.

\bibitem{N}
Norton S.P., Generalized moonshine, in The {A}rcata {C}onference on
 {R}epresentations of {F}inite {G}roups ({A}rcata, {C}alif., 1986),
 \textit{Proc. Sympos. Pure Math.}, Vol.~47, Amer. Math. Soc., Providence, RI,
 1987, 208--210.

\bibitem{N3}
Norton S.P., Moonshine-type functions and the {CRM} correspondence, in Groups
 and symmetries, \textit{CRM Proc. Lecture Notes}, Vol.~47, Amer. Math. Soc.,
 Providence, RI, 2009, 327--342.

\bibitem{Shi}
Shimura G., Introduction to the arithmetic theory of automorphic functions,
 \textit{Publications of the Mathema\-ti\-cal Society of Japan}, Vol.~11,
 Princeton University Press, Princeton, NJ, 1994.

\end{thebibliography}
\end{document}